\documentclass[10pt]{article}

\usepackage{amsmath,amssymb,amsthm,color,graphicx,mathrsfs,amsbsy}
\usepackage{dsfont,amsfonts,bbm}
\usepackage{amsbsy,empheq}
\usepackage{mathrsfs}   
\usepackage{titlesec}
\usepackage{MnSymbol}

\titleformat{\subsection}{\normalfont\large\bfseries}{\thesubsection}{0.7em}{}[]
\titlespacing*{\subsection}{0pt}{2ex plus 0.2ex minus 0.1ex}{1.5ex plus 0.2ex minus 0.1ex}
\titleformat{\subsubsection}[runin]{\normalfont\bfseries}{\thesubsubsection}{0.6em}{}[.]
\titlespacing*{\subsubsection}{0pt}{1.2ex plus 0.2ex minus 0.1ex}{0.6em}

\newtheorem{theorem}{Theorem}

\theoremstyle{remark}

\newcounter{spslist}

\newcommand{\mat}[5]{ \renewcommand{\arraystretch}{#1}
                    \left[\! \begin{array}{cc}
                            #2 & #3 \\
                            #4 & #5 \end{array} \!\right] }

\newcounter{geqncount}
    {\refstepcounter{equation}%
     \setcounter{geqncount}{\value{equation}}%
     \setcounter{equation}{0}%
  }%
    {\setcounter{equation}{\value{geqncount}}}

\newcommand{\eps}{\epsilon}
\newcommand{\Ds}{D_\mathrm{\hspace{-0.7pt}s}}
\newcommand{\Dg}{D_\mathrm{\hspace{-0.7pt}g}}

\newcommand{\FF}{{\mathcal F}}
\newcommand{\EE}{{\mathcal E}}
\newcommand{\BB}{{\mathcal B}}
\newcommand{\MM}{{\mathcal M}}
\newcommand{\UU}{{\mathcal U}}

\newcommand{\ZZ}{\mathbb{Z}}
\newcommand{\RR}{\mathbb{R}}
\newcommand{\PP}{\mathbb{P}}
\newcommand{\CC}{\mathbb{C}}
\newcommand{\TT}{\mathbb{T}}

\newcommand{\VV}{{\mathcal{V}}}
\newcommand{\cT}{{\mathcal{T}}}
\newcommand{\cd}{\hspace{-1.5pt}\cdot\hspace{-2.5pt}}

\newcommand{\NN}{{\mathbb N}}
\newcommand{\calA}{{\mathcal A}}
\newcommand{\calN}{{\mathcal N}}
\newcommand{\scrE}{{\mathscr E}}
\newcommand{\scrF}{{\mathscr F}}

\newcommand{\vol}{\textrm{vol}}

\newcommand{\mapsbackto}{\text{\raisebox{4.5pt}{\rotatebox{180}{$\mapsto$}}}}

\makeatletter

\newcommand{\adj}[1]{#1^\dagger}
\DeclareMathOperator{\tr}{tr}
\DeclareMathOperator{\rk}{rk}
\DeclareMathOperator{\spec}{Spec}
\DeclareMathOperator{\supp}{supp}

\DeclareMathOperator{\diag}{diag}

\begin{document}

\bibliographystyle{plain} 

\begin{center}
{\bf \Large  Algebraic Aspects of Periodic Graph Operators}
\end{center}

\vspace{0.2ex}

\begin{center}
{\scshape \large Stephen P\hspace{-2.5pt}. Shipman${^*}$ \,and\, Frank Sottile${^\dagger}$} \\
\vspace{1ex}
{\itshape ${^*}$Department of Mathematics, Louisiana State University, Baton Rouge, LA 70803} \\
{\itshape Corresponding author, e-mail: shipman@lsu.edu} \\
\vspace{1ex}
{\itshape ${^\dagger}$Department of Mathematics, Texas A\&M University, College Station, TX 77843} \\
{\itshape e-mail: sottile@tamu.edu}
\end{center}

\vspace{3ex}
\centerline{\parbox{0.9\textwidth}{
{\bf Abstract.}\
A periodic linear graph operator acts on states (functions) defined on the vertices of a graph equipped with a free translation action.
Fourier transform with respect to the translation group reveals the central spectral objects, the Bloch and
Fermi varieties.
These encode the relation between the eigenvalues of the translation group and the eigenvalues 
of the operator.
As they are algebraic varieties, algebraic methods may be used to study the spectrum of the operator.
We establish a  framework in which commutative algebra directly comes to bear on the spectral theory of periodic
operators, helping to distinguish their algebraic and analytic aspects.
We also discuss reducibility of the Fermi variety and non-degeneracy of spectral band edges.
}}

\vspace{3ex}
\noindent
\begin{mbox}
\end{mbox}
\vspace{3ex}

\hrule
\vspace{1.1ex}

\section{Introduction} 

Periodic operators on graphs serve as ``tight-binding" models for the quantum mechanics of electrons in crystalline solids.  Their translational symmetry and discrete nature make them naturally amenable to analysis through commutative algebra.
By browsing works such as \cite{DoKuchmentSottile2019a,FaustSottile2023a,FillmanLiuMatos2022,FisherLiShipman2021,GiesekerKnorrerTrubowitz1993}, the reader encounters various aspects of the algebraic nature of these operators and the role of algebra in their spectral analysis; algebraic tools also feature abundantly in periodic PDEs~\cite{Kuchment2016,Kuchment1993}.
This article is a systematic algebraic treatment of periodic graph operators that provides a unifying framework and serves as background material with the hope of benefiting efforts to apply commutative algebra to spectral theory.  The article treats only graphs with finite-degree vertices.
When reading, it is helpful but not necessary to have familiarity with periodic operators, whether discrete or differential.  We refer the reader to introductory sections of the references above, plus an overview focusing on differential operators~\cite{Kuchment2016}.

First, we develop algebraic aspects of the Fourier transform with respect to a discrete translational symmetry group acting freely on a set of vertices.
Upon this, we build the theory of operators that commute with translational symmetries, which are by definition the periodic ones.  Then, half of the article is spent on the Bloch and Fermi algebraic varieties, which describe the relation between momentum (eigenvalues of the translation group) and energy (eigenvalues of the operator).  The two most salient issues are reducibility and nondegeneracy, both of which impact the spectral theory of the operator.  Throughout, we pay attention to identifying algebraic versus analytic aspects of the material.
A general background on the algebraic and geometric ideas we use is found in the two books~\cite{CLOII,CLOI} with more specialized
background in~\cite{CLS,ShafI,ShafII}.

\section{$\ZZ^d$ actions and the Fourier transform}

The algebraic manipulations of \S\ref{silly} set the framework for viewing discrete periodic operators in the context of commutative algebra.
They illuminate various points of view and elucidate the roles of algebra and of analysis.
Algebraic and analytic points of view of the Fourier transform are discussed in~\S\ref{inversion}.

\subsection{Algebraic structure}\label{silly}

We examine the underlying structure of a set of vertices with a translation group that extends to functions defined on the vertices.

\subsubsection{Translation of vertices}
Let $\{e_1,\dots,e_d\}$ be a basis for the group $\ZZ^d$, with a general element denoted by $n=\sum_{i\in[1,d]}n_i e_i$ ($n_i\in\ZZ$).  We use the notation
\begin{equation}
  [1,d] \;=\; \{1, \dots,d\}
\end{equation}
throughout this exposition.  A free action of $\ZZ^d$ on a set $\VV$ of ``vertices" is denoted by
\begin{equation}\label{action}
  \ZZ^d\!\times\!\VV \,\longrightarrow\,\VV \;::\; (n,x) \longmapsto x\dotplus n,
\end{equation}
and we will also use the notation $x\dotminus n:=x\dotplus -n$.  A free action by $\ZZ^d$ on $\VV$ is called a group of (abstract) shifts or translations, and the dot in the notation $\dotplus$ is meant to distinguish an abstract action from a concrete addition such as in (\ref{concrete}) below.  Let $W\!\subset\!\VV$ be a fundamental domain for the action, that is, $W$ contains one vertex from each $\ZZ^d$ orbit.  Because the action is free, each $x\!\in\!\VV$ has a unique presentation as $x = y\dotplus n$ with $y\in W$ and $n\in\ZZ^d$, resulting in a bijection
\begin{equation}\label{union}
    \VV\,\longleftrightarrow\, \ZZ^d\times W \;::\; y\dotplus n \,\longleftrightarrow\, (n,y).
\end{equation}
We assume that $W$ is finite so that the $\ZZ^d$ action is co-finite.
When $|W|=1$, we have the natural action of $\ZZ^d$ on itself by addition, $x\dotplus n = x+n$.
We identify $\VV$ with the disjoint union of $W$ copies of $\ZZ^d$, and $W$  with the $W$ copies of $0\!\in\!\ZZ^d$,
with $\ZZ^d$ acting by addition on each copy of itself.

Concretely, $\VV$ could be a periodic set of points in~$\RR^d$ generated by a finite set of points $W\!\subset\!\RR^d$ through their shifts by $d$ independent vectors $\{v_j : j\in[1,d]\}$.  In this situation, the $\ZZ^d$ action~is actual addition in~$\RR^d$, 
\begin{equation}\label{concrete}
  x \dotplus n \;:=\;  x +\!\! \sum_{j\in[1,d]}\!n_j v_j.
\end{equation}
A tight-binding model for single-layer graphene has $|W|=2$ with three repeating edges.
\[
  \begin{picture}(86,60)
    \put(0,0){\includegraphics{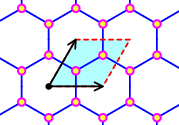}}
    \put(37,32){\small$W$}
  \end{picture}
\]

\subsubsection{Translation of functions}
The group algebra $\CC[\ZZ^d]$ is the complex vector space having $\ZZ^d$ as a basis, and endowed with the natural multiplication
\begin{equation}\label{convolution}
  \sum_m a_mm\cdot \sum_n b_n n
  \;=\;  \sum_{m,n} a_m b_n (m+n) \;=\; \sum_\ell\Big( \sum_m a_m b_{\ell-m} \Big)\ell \;=\; \sum_\ell (a\ast b)_{\hspace{-1pt}\ell}\,\ell\,,
\end{equation}
in which all sums are over $\ZZ^d$ and have only finitely many nonzero terms.
Identifying  $\sum a_n n$ with the function $n\mapsto a_n$ identifies  $\CC[\ZZ^d]$  with the space $F_0(\ZZ^d)$ of complex-valued functions on $\ZZ^d$
having finite support.
By~\eqref{convolution} this identifies the group-algebra multiplication in $\CC[\ZZ^d]$ with convolution in~$F_0(\ZZ^d)$.

We often consider $a=\sum a_n n\in\CC[\ZZ^d]$ as acting on $b=\{b_\ell\}_\ell\in F_0(\ZZ^d)$.
The restriction of this multiplication to the group $\ZZ^d\subset\CC[\ZZ^d]$ is the shift, or translation, action on $F_0(\ZZ^d)$,
\begin{equation}\label{shift}
  (m\cdot b)_\ell \;=\; b_{\ell-m},
\end{equation}
and convolution is a linear combination of shift operators.
These actions on $F_0(\ZZ^d)$ extend to the vector space  $F(\ZZ^d)=\CC^{\ZZ^d}$ of all complex-valued functions on $\ZZ^d$---the function is simply shifted by~$m$ as in~(\ref{shift}).
Restricting to square-summable functions is a unitary representation of $\ZZ^d$ on~$\ell^2(\ZZ^d)$.

This shift action is carried to $F(\VV)=\CC^\VV$, and the unitary representation is extended to $\ell^2(\VV)$~by
\begin{equation}
  (n\cd f)(x) \;=\; f(x\dotminus n)
\end{equation}
(the notation $\dotminus$ is described after~(\ref{action})).  
The identification~\eqref{union} shows that $F(\VV)$ is vector-valued functions on~$\ZZ^d$.  There are various ways to realize this view of $F(\VV)$ or any of its sub-vector-spaces; for example,
\begin{equation}
  \ell^2(\VV) \;\cong\; \ell^2(\ZZ^d)\otimes \CC^W \;\cong\; \ell^2(\ZZ^d,\CC^W)\;\cong\; \ell^2(\ZZ^d)^W.
\end{equation}
The shift operator by the generator $e_j$ of $\ZZ^d$ is also denoted by~$S_j$,
\begin{equation}
  (S_j f)(x) \;=\; f(x\dotminus e_j).
\end{equation}
%

\subsubsection{Formal Fourier transform}
The additive group $\ZZ^d$ is isomorphic to the multiplicative group of monomials $\MM$ in $d$ indeterminates $z=(z_1,\dots,z_d)$,
and the convenient (for our context) isomorphism is
\begin{equation}
  \ZZ^d \,\longleftrightarrow\, \MM \;:\; n \,\longleftrightarrow\, z^{-n},
\end{equation}
in which $z^n = \prod z_i^{n_i}$.  Then one has the identifications
\begin{equation}
  F_0(\ZZ^d) \;\cong\; \CC[\ZZ^d] \;\cong\; \CC[\MM] \;=\; \CC[z^\pm],
\end{equation}
with $\CC[z^\pm]$  being the ring of Laurent polynomials in $d$ indeterminates.
Therefore we obtain the identifications
\begin{equation}
  F_0(\VV) \;\cong\; F_0(\ZZ^d)^W \;\cong\; \CC[\ZZ^d]^W \;\cong\; \CC[z^\pm]^W,
\end{equation}
in which the last object is a finitely generated free module over~$\CC[z^\pm]$.

The algebraic isomorphism from $F_0(\ZZ^d)$, with convolution multiplication, to the ring $\CC[z^\pm]$ of Laurent polynomials in $d$~variables is denoted by~$\hat{}$\,,
\begin{equation}\label{fourier1}
  f \;\longmapsto\; \hat f \;=\; \sum_{n\in\ZZ^d} f(n)z^{-n} \;=\; \sum_{n\in\ZZ^d} f(-n) z^n.
\end{equation}
This is a {\itshape formal Fourier transform}, sometimes called the (formal) {\itshape Floquet transform} or {\itshape $z$-transform}.
It is ``formal" because $z$ is an indeterminate rather than a variable to be evaluated at a point.
The extension of the Floquet transform to $F(\ZZ^d)$ yields $\hat f$ in the module $\CC[[z^\pm]]$ of formal Laurent series, and one has
\begin{equation}
  f\in F(\ZZ^d) \text{ has finite support } \iff \hat f  \text{ is a Laurent polynomial}.
\end{equation}

To extend this transform to~$F(\VV)=\CC^\VV$, we fix a fundamental domain $W$ and identify $F(\VV)$ with $F(\ZZ,\CC^W)$:
Consider $f\in F(\VV)$ as a function of $(n,y)\in\ZZ\times W$ through the unique presentation $x=y\dotplus n$; then as a function of the variable $n$,
the object $f(\bullet\dotplus n)$ is an element of~$\CC^W$.
The formula (\ref{fourier1}) still holds, with the values of $f$ being vector rather than scalar,
\begin{equation}\label{fourier2}
  f \;\longmapsto\; \hat f \;=\; \sum_{n\in\ZZ^d} f(\bullet\dotplus n)z^{-n} \;=\; \sum_{n\in\ZZ^d} f(\bullet\dotminus n) z^n.
\end{equation}
This is a Laurent series in the indeterminate $z$ with coefficients in~$\CC^W$.
One can eliminate the need to fix a fundamental domain simply by allowing the argument ``$\bullet$" to run over all of $\VV$ and not just~$W\!$.  This yields a Laurent series in the indeterminate $z$ with coefficients in~$F(\VV)$, 
\begin{equation}\label{floquet2}
  \hat f(z,x) \;=\; \sum_n f(x\dotplus n) z^{-n}.
\end{equation}
This is a sum over the $\ZZ^d$-orbit of the point $x\in\VV$.
As the coefficients are just shifts of~$f$, the formal Floquet transform on $F(\VV)=\CC^\VV$ can be written more compactly~as
\begin{equation}\label{floquet1}
  f \;\;\longmapsto\;\; \hat f \;=\; \sum_{n\in\ZZ^d} n\cd f\, z^n.
\end{equation}
The advantage of retaining $x$ in the Floquet transform is that the action of $n\in\ZZ^d$ on $\hat f(z,\bullet)$ coincides with the action of $z^{-n}\in\MM$ on $\hat f(z,\bullet)$,
\begin{equation}\label{fproperty}
  n\cd\hat f \;=\; z^{-n} \hat f.
\end{equation}

\subsection{$L^2$ Floquet transform and Fourier inversion}\label{inversion}

Denote the multiplicative group of complex units by $\CC^\times = \CC\!\setminus\!\{0\}$.
If the Fourier sum (\ref{fourier1}) converges when evaluated at $z\!\in\!(\CC^\times)^d$, the symbols $z_1,\dots,z_d$ become complex variables rather than formal algebraic symbols, and $\hat f(z,x)$ becomes a function of two arguments.

\subsubsection{Eigenfunctions of the translation group}
If (\ref{fourier1}) converges for some $z\!\in\!(\CC^\times)^d$, then the fundamental property (\ref{fproperty}) of the Floquet transform says that $\hat f(z,x)$, viewed as a function of $x\in\VV$ is an eigenfunction of the $\ZZ^d$ action, with $n$-dependent eigenvalue equal to the homomorphism
\begin{equation}
  \ZZ^d \,\longrightarrow\, \CC^\times \;::\; n \,\longmapsto\, z^{-n}.
\end{equation}
The property can be written equivalently as
\begin{equation}
  \hat f(z,x\dotplus n) \;=\; z^n \hat f(z,x)
  \qquad \forall n\in\ZZ^d.
\end{equation}
Often this direction of the shift ($f(x\dotplus n)$ instead of $f(x\dotminus n)$) is taken to be the definition of the shift action, so that $z^n$ (instead of $z^{-n}$) is the eigenvalue of the shift by $n\!\in\!\ZZ^d$.

All homomorphisms from $\ZZ^d$ to $\CC^\times$ form the dual group of {\itshape characters} of $\ZZ^d$.
Each character corresponds to a $d$-tuple $z=(z_1,\dots,z_d)$ of nonzero complex numbers, called its {\itshape weight}.
Thus the dual group is isomorphic to the non-compact complex torus~$(\CC^\times)^d$.
The unitary dual group of $\ZZ^d$ is the group of homomorphisms from $\ZZ^d$ to the unit circle $\TT\!\subset\!\CC$,
which is isomorphic to the compact $d$-dimensional torus~$\TT^d\!\subset\!(\CC^\times)^d$.
Reciprocally, $\ZZ^d$ is the dual group of the torus (continuous homomorphisms from $(\CC^\times)^d$ to $\CC^\times$).
The  bihomomorphism $z^{-n}$ captures both directions,
\begin{equation}
  \ZZ^d \times (\CC^\times)^d \;\longrightarrow\; \CC^\times \;::\; (n,z) \longmapsto z^{-n}.
\end{equation}

Denote by $\EE_z$ the eigenspace of the $\ZZ^d$ action on $F(\VV)\cong F(\ZZ^d)\otimes\CC^W$ for the weight~$z^{-1}$.  Because $f(y\dotplus n)=z^n f(y)$, each $f\in\EE_z$ is determined uniquely by its restriction to $W$, and thus, as vector spaces,
\begin{equation}\label{EEiso}
  \EE_z \;\cong\; \CC^W.
\end{equation}
Explicitly, an element $g\in\CC^W$ corresponds to an element $Q(g,z)\in\EE_z$ by
\begin{equation}
  Q(g,z)(n) \;=\; g z^n,
\end{equation}
considered as an element of $F(\ZZ^d)\otimes\CC^W$ (or as an element of $F(\VV)$, one would write $Q(g,z)(v)=g(w)z^n$, where $v=n\dotplus w$ with $w\in W$).  The symbol $Q$ stands for ``quasi-periodic".

For each $z\!\in\!(\CC^\times)^d$, let $\varepsilon_z$ be a distinguished element of $\EE_z$ such that $\varepsilon_z(x)\not=0$ for all~$x\in\VV$.  Then each $f\in\EE_z$ can be written as
\begin{equation}
  f(x) \;=\; \tilde f(x) \varepsilon_z(x).
\end{equation}
The function $\tilde f$ is called the {\itshape periodic part} of $f$ (with respect to the choice of~$\varepsilon_z$).  For example, one can take $\varepsilon_z(y)=1$ for all $y\in W$ so that $\varepsilon_z(y\dotplus n)=z^n$.  When $\VV$ is a periodic subset of $\RR^d$, as in Equation~\eqref{concrete},
let $q=(q_1,\dots,q_d)$ be the dual (momentum) variable so that $k_j=q\cdot v_j$ and $z_j=e^{i q\cdot v_j}$.  Then, for $x\in\RR^d$, one can define $\eps_z(x)=e^{i q\cdot x}$, and restrict this function to $\VV\subset\RR^d$.

\subsubsection{Fourier inversion} 
We have discussed that $F_0(\VV)$ is isomorphic to $\CC[z^\pm]^W$ through the formal Floquet transform.  By evaluating $\hat f$ at actual points $z\mapsto e^{i k}\!\in\!(\CC^\times)^d$, it becomes a function whose restriction to the compact torus $\TT^d$ is a trigonometric polynomial in $k\in\RR^d/(2\pi\ZZ)^d$,
\begin{equation}\label{trigpoly}
  \hat f(e^{i k}) \;=\; \sum_{n\in\supp(f)}\! f(-n) e^{i k\cdot n}
  \qquad \text{(finite sum)}.
\end{equation}
In this notation, $\hat f(z)$ means $\hat f(z,\bullet)$.
After evaluation on $\TT^d$, $\CC[z^\pm]^W$ becomes the ring of trigonometric polynomial functions $\cT(\TT^d)\otimes\CC^W$ on the compact $d$-torus, which have the form~(\ref{trigpoly}).  We denote the map from $f\in F_0(\VV)$ to $\hat f(e^{i k})$ by~$\UU$,
\begin{equation}
  \UU : F_0(\VV) \;\cong\; F_0(\ZZ^d)\otimes\CC^W \;\longrightarrow\; \cT(\TT^d)\otimes\CC^W.
\end{equation}
$\UU f$ could also be denoted by $\hat f$, with the understanding that a formal polynomial $\hat f(z)$ is being evaluated.

An $L^2$ norm on $\EE_z$ is defined~by 
 $ \|f\|_2^2 = \sum_{y\in W}|f(y)|^2 $.
A fundamental theorem of Fourier analysis says that $\UU$ is an $L^2$ isomorphism,
\begin{equation}
  \sum_n |f(n)|^2 \;=\; \int_{\TT^d} \big|\UU f(z)\big|^2\, d\tilde V(z),
\end{equation}
in which $d\tilde V(z) = d k_1\cdots d k_d/(2\pi)^d$, with inverse given~by
\begin{equation}\label{FInv1}
  f(n) \;=\; \int_{\TT^d}\hat f(z,n)\, d\tilde V(z).
\end{equation}
In terms of the periodic part $\tilde f$ of $\hat f$, this is the familiar inverse Fourier transform with kernel $\varepsilon_z(x)$,
\begin{equation}\label{FInv2}
    f(y\dotplus n) \;=\; \int_{\TT^d}\tilde f(y,z)\,\varepsilon_z(y\dotplus n) \, d\tilde V(z).
\end{equation}

\subsubsection{Completion to $L^2$}
So far, we have defined the formal (algebraic) and analytical Fourier transforms on functions of compact support in~$\VV$,
\begin{equation}
  f\in F_0(\VV) \;\longrightarrow\; \hat f(z) \in \CC[z^\pm]^W \;\longrightarrow\; \UU f\in \cT(\TT^d)\otimes\CC^W.
\end{equation}
Completion in the root-mean-square norm yields the unitary Fourier transform of Hilbert spaces,
\begin{equation}
  \UU : \ell^2(\VV) \longrightarrow L^2(\TT^d)\otimes\CC^W.
\end{equation}
The inversion formulas (\ref{FInv1}) and (\ref{FInv2}) state that the formula (\ref{floquet2}) for the Floquet transform produces the components in all the $\EE_z$ required to build~$f$.
Fourier inversion amounts to writing $\ell^2(\VV)$ as a direct integral over the torus of simpler spaces~\cite[\S{XIII.16}]{ReedSimon1980d},
\begin{equation}\label{directintegral1}
  \ell^2(\VV) \;=\; \int^\otimes_{\TT^d} \EE_z\, d\tilde V\!(z),
  \qquad
  f \;=\; \int^\otimes_{\TT^d} \hat f(z,\bullet) \, d\tilde V\!(z),
\end{equation}
and the transformation $\hat f\mapsto\tilde f$ is a simple gauge transformation that transforms all the fibers $\EE_z$ to~$\EE_{\mathbf{1}}$. 

\def\cprime{$\,'\!$}
\subsubsection{Spec and spectrum}
Evaluation of the formal Floquet transform at values in the torus is analogous to the {\itshape Gel\cprime fand transform} for  Banach algebras.
Evaluation of $\hat f(z)\in\CC[z^\pm]$ at $z\mapsto\zeta=(\zeta_1,\dots,\zeta_d)\in(\CC^\times)^d$ sends the polynomial $\hat f(z)$
to the complex number $\hat f(\zeta)$, which is the representative of the coset $\hat f(z)\,\mathrm{mod}\langle z-\zeta \rangle$ in the field
$\CC[z^\pm]/\langle z-\zeta \rangle\simeq\CC$.
Formally, $\hat f$ becomes a function on the space of maximal ideals in $\CC[z^\pm]$, called $\spec\CC[z^\pm]$, which is homeomorphic to~$(\CC^\times)^d$,
\begin{equation}
  (\CC^\times)^d \;\cong\; \spec\CC[z^\pm]  \;\;::\;\; \zeta \longleftrightarrow \langle z-\zeta \rangle.
\end{equation}
Equivalently, $\spec\CC[z^\pm]$ is the set of ring homomorphisms from $\CC[z^\pm]$ to~$\CC$.
The homomorphism that takes the indeterminate $z_j$ to~$\zeta_j$ for $j\in[1,d]$ has kernel equal to the ideal $\langle z\!-\!\zeta \rangle$
and sends the polynomial $\hat f(z)$ to the number $\hat f(\zeta)$.

As we have seen, translation of compactly supported functions by $n$, that is, the operator $S^n\!=\!\prod_{j\in[1,d]}S_j^{n_j}$ becomes multiplication of Laurent polynomials by $z^{-n}$, which in turn becomes multiplication of trigonometric polynomials on $\TT^d$ by the function~$\zeta^{-n}$.
This means that, in the direct integral~\eqref{directintegral1}, the fiber at $\zeta\!\in\!\TT^d$ of the operator $S^n$, acting on the fiber $\CC^W$ at $\zeta$ is just $\zeta^{-n} I_W$.  Particularly, the elementary shift $S_j$ has fiber at $\zeta$ equal to the coordinate~$\zeta_j$.
This is because the spectrum of $S_j$ is the unit circle, being the range of the $j$-th coordinate function $\zeta_j$ on~$\TT^d$.  In this way, the unitary part of $\spec\CC[z^\pm]$ gives the spectrum of the shift group.
More generally, a polynomial in the elementary shifts, with matrix coefficients,
\begin{equation}\label{periodic0}
  A(S) \;=\; \sum_{n\in\ZZ^d} A_n S^n
\end{equation}
has spectrum equal to the image of the multi-valued eigenvalue function
\begin{equation}
  \sigma(\zeta) \;=\; \left\{ \lambda : \det(A(\zeta)-\lambda)=0 \right\}
\end{equation}
over the compact torus~$\TT^d$.
The Laurent polynomial $\det(A(\zeta)\!-\!\lambda)$ and its zero set predominate the $L^2$ spectral theory of the operator $A(S)$.  An operator that commutes with the shift group is called {\itshape periodic}, and these are discussed in the next section.

\section{Periodic graph operators}\label{sec:periodic}

An operator $A$ acting on functions on $\VV$, endowed with a free $\ZZ^d$ shift action, is periodic if it commutes with the shift group; such an operator is said to possess translational symmetry.
Periodic operators in $\ell^2(\VV)$ have a dichotomy of spectrum: eigenvalues of infinite multiplicity and bands of continuous spectrum.
Eigenvalues are completely described by $A$ as a module endomorphism, while continuous spectrum comes from the Bloch variety, which is the singular set of the matrix $\hat A(z)-\lambda I$.
An interesting exposition of the module homomorphism point of view with extension from $\ZZ^d$ to amenable group actions is offered in~\cite{Kravaris2023a}.  Amenability of a group action has also been exploited for characterizing the eigenvalues for percolation Hamiltonians on graphs in~\cite{Veselic2005a}.

\subsection{On $F_0(\VV)$: Module endomorphism}

A periodic operator $A$ restricted to functions of compact support is a module endomorphism of~$\CC[z^\pm]^W$.

\subsubsection{Definition and notation}
Let $A$ be a linear operator on $F_0(\VV)$, and denote by~$\hat A$ the corresponding operator on $\CC[z^\pm]^W$ as a vector space over $\CC$, through the Floquet transform.  That is, $\hat A$ is defined through
\begin{equation}
  (\hat A \hat f)(z) \;=\; 
                          \widehat{(A f)}(z).
\end{equation}
As $F_0(\VV)\cong F_0(\ZZ^d)\otimes\CC^W$ consists of functions of finite support on $\ZZ^d\!\times\!W$, it is convenient to denote the evaluation of $f\in F_0(\VV)$ at $(n\in\ZZ^d,v\in W)$ by $f_v(n)$, so that $f(n)$ refers to the vector $(f_v)_{v\in W}$ in~$\CC^W$.  Being linear, $A$ is associated with its matrix $a(m,n)$ through
\begin{equation}
  (A f)(m) \;=\; \sum_{n\in\ZZ^d} a(m,n)f(n).
\end{equation}
For each $m,n\in\ZZ^d$, $a(m,n)$ is a linear operator on $\CC^W$ and is thus represented by a matrix $(a_{v w}(m,n))_{v,w\in W}$ indexed by $W$.  This means
\begin{equation}\label{E:operator_first}
  (A f)_v(m) \;=\; \sum_{n\in\ZZ^d} \sum_{w\in W} a_{v w}(m,n)f_w(n).
\end{equation}
The operator $A$ is said to have {\itshape finite range of interaction} (or simple finite range) if there exists a number $R$ such that, for all $m,n\in\ZZ^d$, $a(m,n)\!=\!0$ whenever $|m-n|>R$.
If $a(n,m)=a(m,n)^*$, then $A$~is self-adjoint.  Here, $(\ )^*$ is the adjoint with respect to the standard inner product of~$\CC^W$.

The linear operator $A$ is said to be {\itshape periodic} if, for all $f\in F_0(\ZZ^d)\otimes\CC^W$,
\begin{equation}\label{periodic1}
  (A f)(\bullet\dotplus n) \;=\; A[f(\bullet\dotplus n)]
  \qquad\forall n\in\ZZ^d.
\end{equation}
More compactly, this means that $A$ commutes with the shift group~$\ZZ^d$,
\begin{equation}\label{periodic2}
  n\cd A f \;=\; A(n\cd f)
    \qquad\forall n\in\ZZ^d.
\end{equation}
Observe that $A$ is periodic if and only if
\begin{equation}
  a(m,n) \;=\; \tilde a(m-n)
\end{equation}
for some matrix function $\tilde a$, which we will just denote by~$a$, or $a_{vw}(n)$ in full.  A~periodic linear operator of finite range has $a(m)=0$ whenever $|m|>R$; denote by $r(A)$ the minimal such $R$\label{Arange}.   The operator is self-adjoint if $a(-m)=a(m)^*$.  This means that $A$ is a convolution operator
\begin{equation}
  (A f)(m) \;=\; \sum_{n\in\ZZ^d}a(m-n) f(n) \;=\; \sum_{n\in\ZZ^d} a(n)f(m-n),
\end{equation}
or, equivalently, a linear combination of translations with matrix coefficients,
\begin{equation}
  A \;=\; \sum_{n\in\ZZ^d} a(n)\, n\cdot \;=\; \sum_{n\in\ZZ^d} a(n)S^n,
\end{equation}
as in~(\ref{periodic0}).
When $W$ consists of only one vertex, $A$ is an element of the group algebra $\CC[\ZZ^d]$ and is sometimes called a Toeplitz operator.  Now identify $f$ with $\sum f(n) n\in\CC[\ZZ^d]^W$ and use (\ref{convolution}) to obtain
\begin{equation}\label{convolution2}
  \sum a(n) n \cdot \sum f(n) n \;=\; \sum (A f) (n) n.
\end{equation}

\subsubsection{Graph of an operator}\label{graph}
  A linear operator $A$ on $F_0(\VV)$ corresponds to a directed, labeled graph $\Gamma=(\VV, E)$.
  Identifying $\VV$ with $\ZZ^d\!\times\!W$, there is a directed edge $v{\dotplus}m{\leftarrow}w\dotplus n$ in $E$ with
  label $a_{v w}(m,n)$, when $a_{v w}(m,n)\neq 0$.
  Loops $v\dotplus m \rcirclearrowleft$ are  removed and their labels $a_{v v}(m,m)$ become  a potential
  function $V\colon \VV\to\CC$.
  For an edge $v{\leftarrow}w$, we will write $a_{v\leftarrow w}$ for the corresponding edge label.
  Then, for $f\in F_0(\VV)$, $v,w\in\VV$~\eqref{E:operator_first} becomes
  \begin{equation}\label{Eq:graphOperator_first}
    (A f)(v) \;=\; V(v)f(v)\ +\ \sum_{v\leftarrow w} a_{v\leftarrow w} f(w).
  \end{equation}

 When $A$ has finite range, the graph $\Gamma$ is locally finite in that each vertex $v$ has finitely many neighbors
 $\{w\in\VV\mid v\leftarrow w\}$.
 By~\eqref{Eq:graphOperator_first},  $A$ is a Schr\"odinger operator: a potential plus weighted graph
 Laplacian, and it acts on $F(\VV)$.

 If in addition $A$ is $\ZZ^d$-periodic, then $\Gamma$ is a $\ZZ^d$-periodic labeled graph.
 E.g.\ the underlying graph is periodic with $V\colon\VV\to\CC$ and $a\colon E\to \CC$ periodic functions.
 In this case, the operator $A$ commutes with the $\ZZ^d$-action, it is {\itshape periodic}.
 When $a_{v w}(m,n)=a_{w v}(n,m)^*$, the operator $A$ is self-adjoint, and when
 $a_{v w}(m,n)=a_{w v}(n,m)$, we may replace the directed graph by a corresponding undirected graph.

\subsubsection{Examples}
 The simplest example is when $\VV=\ZZ$ and there is an (undirected) edge between $n$ and $n\pm 1$ with weight $-1$,
 so that $\Gamma$ is \includegraphics{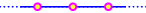} and $A$ becomes
 \begin{equation}\label{Eq:lineGraphOperator}
  (A f)(n)\;=\; -f(n-1) + V(n) f(n) - f(n+1).
 \end{equation}
 Another example is the $\ZZ^2$-periodic hexagonal lattice (graphene), in which $W=\{v,w\}$ and there are (undirected) edges
 connecting $v(=v\dotplus(0,0))$ to each of $w$, $w\dotminus(1,0)$, and $w\dotminus(0,1)$, and this is extended to all of
 $\ZZ^2\!\times\!W$ by periodicity.
 We show the hexagonal lattice with $W$ shaded and the $\ZZ^2$-action indicated by the
 arrows, together with a picture of a labeling of edges in a neighborhood of $W$.
 \begin{equation}\label{Eq:hexagonal_lattice}
   \raisebox{-38pt}{\begin{picture}(115,82)
       \put(0,0){\includegraphics{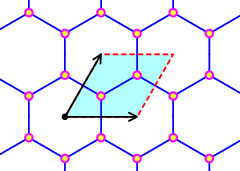}}
       \put(49,44){$W$}
   \end{picture}
   \qquad\qquad
   \begin{picture}(160,82)(-27,3)
     \put(9,4){\includegraphics{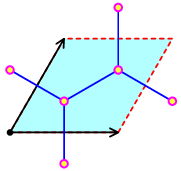}}
     \put(31,31){\small$v$}      \put( 96,35){\small$v\dotplus(1,0)$}     \put(70,81){\small$v\dotplus(0,1)$}
     \put(69,53){\small$w$}      \put(-27,50){\small$w\dotminus(1,0)$}    \put(43,6){\small$w\dotminus(0,1)$}
     \put(68,19){\small$(1,0)$}
     \put(18,70){\small$(0,1)$}
     \put(52,39){\small{\color{blue}$a$}}
     \put(33,15){\small{\color{blue}$c$}}   \put(58,72){\small{\color{blue}$c$}}
     \put(16,41){\small{\color{blue}$b$}}   \put(86,44){\small{\color{blue}$b$}}
   \end{picture}}
 \end{equation}
 The operator $A$ for the hexagonal graph with edge-labeling $a,b,c$ as above is
 \begin{equation}\label{Eq:hexagonal_lattice_operator}
   \begin{array}{rcl}
    (A f)(v\dotplus n) &=& V(v) f(v\dotplus n) \ +\ a f(w\dotplus n) + b f(w\dotplus n \dotminus(1,0)) + c f(w\dotplus n \dotminus(0,1)) \\
    (A f)(w\dotplus n) &=& V(w) f(w\dotplus n) \ +\ a f(v\dotplus n) + b f(v\dotplus n  \dotplus(1,0)) + c f(v\dotplus n  \dotplus(0,1)) 
   \end{array}\ .
 \end{equation}

\subsubsection{Formal Fourier transform}
Under the formal Fourier (or Floquet) transform $n\mapsto z^{-n}$, translation by $n$ becomes multiplication by $z^{-n}$,
and periodicity becomes ring homomorphism.  More precisely, with
\begin{equation}
  \hat A(z) = \sum a(n) z^{-n}, \qquad
  \hat f(z) = \sum f(n) z^{-n}, \qquad
  \widehat{A f}(z) = \sum (A f)(n) z^{-n},
\end{equation}
formula~\eqref{convolution2} becomes $ \hat A(z) \hat f(z) = (\widehat{A f})\,(z)$.
This is an algebraic form of the general principle that an operator commuting with a group action becomes a multiplication operator under (formal) Fourier transform.

In summary: The space $F_0(\ZZ^d)$, identified with $R\vcentcolon=\CC[z^\pm]$ through the formal Floquet transform, is a vector space over $\CC$ and also a ring (convolution of functions $=$ multiplication of polynomials).
Thus $F_0(\VV)=R^W$ is a free $R$-module.
Consequently, one can study general $\CC$-linear operators on $R^W\!$, or specialized ones that are also $R$-module endomorphisms of~$R^W\!$.
The operator $A$, when pulled over to $R^W\!$, is written~$\hat A$.
The periodicity property~\eqref{periodic2} means that $\hat A$ commutes with multiplication by $z^n$ for all $n\in\ZZ^d$.
This happens if and only if $\hat A$ commutes with multiplication by any $\hat f(z)\in R$; and this is simply the statement that $\hat A$ is a module endomorphism of~$R^W$.
This amounts to the following (un)remarkable result.

\begin{theorem}\label{thm:endomorphism}
  A $\CC$-linear operator $A$ on $F_0(\VV)$ is periodic if and only if $\hat A$ is a module endomorphism of~$\CC[z^\pm]^W$.
\end{theorem}

Any endomorphism $\hat A:R^W\to R^W$ is given by a matrix $\hat a_{v w}(z)$ of elements of $R$ indexed by $W$, that~is
\begin{equation}
  (\hat A p(z))_v \;=\; \sum_{w\in W}\hat a_{v w}(z)p_w(z),
\end{equation}
and one can identify $\hat A(z)$ with its matrix $(\hat a_{v w}(z))_{v,w\in W}$.
For the hexagonal graph, the matrix $\hat{A}(z)$ becomes
 \begin{equation}\label{Eq:hexagonal_Graph_Floquet_Matrix}
   \left(\begin{matrix}     V(v) & a+b z_1^{-1}+c z_2^{-1}\\ a+b z_1+c z_2 & V(w)  \end{matrix}\right)\ ,
 \end{equation}
where $z=(z_1,z_2)$ with $x$ corresponding to $(1,0)$ and $y$ to $(0,1)$, where $v$ indexes the first row and column and $w$ the second row and column, and with the labeling of~\eqref{Eq:hexagonal_lattice}.

\subsubsection{Graph of $A$ on the torus}
The graph $\Gamma$ of the operator $A$ can be drawn in $\RR^d$ in such a way that the $\ZZ^d$ action coincides with
translation by $\ZZ^d\subset\RR^d$.
Under the covering $\pi\colon\RR^d\to\RR^d/\ZZ^d\simeq \TT^d$, this graph is mapped to a graph $\pi\Gamma$ drawn on the
torus with vertex set identified with~$W$.  This graph has an edge between a pair of vertices $\{v,w\}\subset W$ whenever
$\Gamma$ has an edge between $v$ and some translate of~$w$.
As distinct (directed) edges are not homotopic to each other on the torus, therefore the graph $\pi\Gamma$ preserves all the information about~$\Gamma$.  Conversely, any graph on the torus without homotopic edges lifts to a periodic graph in $\RR^d$ without multiple edges.

The {\itshape reduced graph} of $\Gamma$ is a labeled multigraph which is obtained from $\pi\Gamma$ as follows.
It has the same edges $v{\leftarrow}w$ as $\pi\Gamma$, but the homotopy class of an edge $v{\leftarrow}w$ is indicated
by labeling the edge by the monomial $a_{vw}(n)z^{-n}$, where the edge $(v,w+n)$ appears in~$\Gamma$.

\subsubsection{Eigenvalues}\label{sec:eigenvalues}
For $\lambda\in\RR$, the equation $(A-\lambda)f=0$ for $f\in F_0(\VV)$ is equivalent to $(\hat A(z)-\lambda)\hat f(z)=0$,
where $\hat f(z)\in\CC[z^\pm]$.
Thus $\lambda$ is an eigenvalue of $\hat A(z)$.
For all $n\in\ZZ^d$, we have $(\hat A(z)-\lambda)z^n f(z)=0$, so that 
all shifts of $f$ are also eigenvectors of~$A$ and thus $\lambda$ is an eigenvalue of infinite multiplicity.

\subsection{On $F(\VV)$: Floquet eigenspaces}\label{sec:FB}

An operator $A$ of finite range can be applied to functions of infinite support.  Equivalently, if $\hat A(z)$ is a Laurent polynomial in $z=(z_1,\dots,z_d)$ with coefficients that are $W\!\times\!W$ matrices with entries in $\CC$, $\hat A(z)$ can be applied to Laurent series with vector coefficients, that is, to elements of $\CC[[z^\pm]]^W$.

\subsubsection{Action of $A$ on $\EE_\zeta$}
For $\zeta\in(\CC^\times)^d$, nonzero elements of the space $\EE_\zeta$ of quasi-periodic functions have infinite support.  
A periodic operator $A$ is invariant on $\EE_\zeta$, so via the isomorphism $\EE_\zeta\cong\CC^W$ (\ref{EEiso}), the restriction of $A$ to $\EE_\zeta$ induces an operator on $\CC^W$.  This operator is just $\hat A(\zeta)$, the evaluation of $\hat A(z)$ at~$z=\zeta$.  For a function $f\in F(\VV)$, denote $\hat f(z)$ also by $\FF(f)(z)$.  

\begin{theorem}\label{thm:QA}
  Let $A$ be a periodic operator of finite range, and let $g\in\CC^W$ and $\zeta\in(\CC^\times)^d$.  Then
\begin{equation}\label{Eq:thmQA}
  \hat A(z) [\FF Q(g,\zeta)] \;=\; \FF Q(\hat A(\zeta)g,\zeta).
\end{equation}
Equivalently, $A [Q(g,\zeta)]= Q(\hat A(\zeta)g,\zeta)$.
\end{theorem}

\begin{proof}
  By definition, one has
\begin{equation}
  Q(g,\zeta)(n) = g\,\zeta^n,
  \qquad
  \FF Q(g,\zeta)(z) = \sum_{n\in\ZZ^d} g\,\zeta^n\,z^{-n}.
\end{equation}
Let $\hat A(z)= \sum_{\ell\in\Lambda}\!A^\ell z^\ell$ for some finite subset $\Lambda$ of~$\ZZ^d$.
A straightforward sequence of equalities yields
\begin{equation}
\begin{aligned}
  \hat A(z)[\FF Q(g,\zeta)] &\;=\; \hat A(z)\Big( \sum_{n\in\ZZ^d} g \zeta^n z^{-n} \Big)\ =\ \sum_{n\in\ZZ^d} \zeta^n \hat A(z)g z^{-n} 
           \ =\ \sum_{n\in\ZZ^d} \zeta^n\sum_{\ell\in\Lambda} A^\ell g z^{\ell-n} \\
   &\;=\; \sum_{\ell\in\Lambda} A^\ell g \sum_{n\in\ZZ^d} \zeta^n z^{\ell-n} \ =\  \sum_{\ell\in\Lambda} A^\ell g \sum_{n\in\ZZ^d} \zeta^{n+\ell}z^{-n} \\
  &\;=\; \sum_{n\in\ZZ^d} z^{-n}\zeta^n\sum_{\ell\in\Lambda} A^\ell \zeta^\ell g \ =\  \sum_{n\in\ZZ^d} z^{-n} \hat A(\zeta)g \zeta^n
           \ \ =\  \FF Q(\hat A(\zeta)g,\zeta)\,. 
        \makebox[10pt][l]{\mbox{\hspace{50pt}}\qedhere}   
\end{aligned}
\end{equation}
\end{proof}

\subsubsection{Floquet-Bloch theorem}
By Theorem~\ref{thm:endomorphism}, operators that are periodic on $\VV$ with respect to the free $\ZZ^d$ action are exactly
those that become multiplication operators under the formal Floquet transform.
That is, $A f$ becomes $\hat A(z)\hat f(z)$, which is matrix-vector multiplication of Laurent polynomials.
Evaluation of these polynomials at a point $\zeta\in\TT^d$ realizes the Fourier transform, that is
\begin{equation}\label{FBthm1}
  (\UU f)(\zeta)\ =\ \hat A(\zeta)\hat f(\zeta).
\end{equation}
%

\subsection{On $L^2(\VV)$: Full spectral resolution}

When extended to $L^2$, the operator $A$ enjoys all the theory of closed operators and more in Hilbert space.

\subsubsection{Floquet-Bloch theorem revisited}
In $L^2$, Equation~\eqref{FBthm1} is a discrete version of the classical Floquet-Bloch Theorem~\cite[Thm~XIII.97]{ReedSimon1980d}, which represents $A$ as a decomposable operator through the direct integral~(\ref{directintegral1}),
\begin{equation}\label{directintegral2}
  A \;=\; \int^\otimes_{\TT^d} \hat A(z)\, d\tilde V\!(z),
  \qquad
  A f \;=\; \int^\otimes_{\TT^d} \hat A(z) \hat f(z,\bullet) \, d\tilde V\!(z).
\end{equation}
The resolvent of $A$ is represented by
\begin{equation}
    (A-\lambda)^{-1}f \;=\; \int^\otimes_{\TT^d} (\hat A(z)-\lambda I)^{-1} \hat f(z,\bullet) \, d\tilde V\!(z).
\end{equation}
%

\subsubsection{Invertibility and spectrum}\label{sec:invspec}

We do not expect that  $(A-\lambda I)$ is invertible on $F_0(\VV)$.
This is best seen through the Floquet transform:  The matrix operator $(\hat A(z)\!-\!\lambda I_W)$ on $\CC[z^\pm]^W$ is a Laurent polynomial in~$z$, so it takes polynomials to polynomials, but its inverse does not.  Extending $A$ to $\ell^2(\VV)\cong \ell^2(\ZZ^d)\otimes\CC^W$, or, equivalently, extending $\hat A(z)$ to $L^2(\TT^d)\otimes\CC^W$, the inverse as a bounded linear operator~is
\begin{equation}
  (\hat A(z) - \lambda I)^{-1} \;=\; \frac{\adj{A}(z,\lambda)}{D(z,\lambda)},
\end{equation}
whenever this is regular on $\TT^d$.  Here, $\adj{A}(z,\lambda)$ is the adjugate matrix to $\hat A(z)\!-\!\lambda I$ and
\begin{equation}\label{D}
  D(z,\lambda) \;=\; \det (\hat A(z)\!-\!\lambda I).
\end{equation}
The determinant $D(z,\lambda)$ is known as the {\itshape dispersion function} for the operator~$A$.
For those $\lambda$ for which $D(z,\lambda)$ is non-vanishing for all $z\in\TT^d$, the inverse $(\hat A(z) - \lambda I)^{-1}$ exists from $L^2(\TT^d)\otimes\CC^W$ to~$\ell^2(\VV)$.  Thus the resolvent set of $A$~is the open set
\begin{equation}
  \rho(A) \;=\; \left\{ \lambda\in\CC : D(z,\lambda)\not=0\;\forall z\in\TT^d \right\},
\end{equation}
and the spectrum of $A$ is its complement~$\sigma(A)=\CC\!\setminus\!\rho(A)$.

Symmetry of the operator $A$ in the sense that, for all $f,g\in F_0(\VV)$,
\begin{equation}
  (A f,g) \;=\; (f,Ag)
\end{equation}
occurs whenever the matrix for $A$ is Hermitian, that is, $a_{v w}(m,n)=(a_{w v}(n,m))^*$.  The extension of $A$ to $\ell^2(\VV)$ is self-adjoint.  This implies that the matrix $\hat A(z)$ is self-adjoint for all $z\in\TT^d$,
\begin{equation}
  \hat A(z) \;=\; \hat A(z)^* 
  \qquad \forall z\in(\TT^*)^d,
\end{equation}
and this implies the more general reflection property
\begin{equation}\label{reflection1}
  \hat A(\bar z^{-1}) \;=\; \hat A(z)^* 
  \qquad \forall z\in(\CC^\times)^d,
\end{equation}
which can also be checked directly by the construction $\hat A(z)$ from~$A$.
In this situation, the spectrum of $A$ is real, $\sigma(A)\subset\RR$ and, furthermore, $D(z,\lambda)$ is real-valued whenever $z\in\TT^d$ and $\lambda\in\RR$.

\subsubsection{Spectral bands} 
When $A$ is self-adjoint, $D(z,\lambda)=0$ defines a multi-valued real eigenvalue function $\lambda(z)$ on the torus $\TT^d$ with real eigenvalue branches $\lambda_j(k)$ ($z=e^{i k}$), $j\in[1,|W|]$.  The image of all of these branches constitutes the spectrum $\sigma(A)$ of $A$, which consists a union of closed intervals called {\itshape spectral bands}.  The branches of $\lambda$ are called {\itshape band functions}.  Spectral bands are discussed in \S\ref{sec:BlochFermi}.

\subsubsection{Density of states}\label{DoS}
The physical observable known as the density of states can be expressed in terms of the band functions~\cite[Ch.~11]{GiesekerKnorrerTrubowitz1993}.  If $\nu$ is the integrated density of states, the density is
\begin{equation}
  \delta(\lambda) \;=\; \frac{d\nu}{d\lambda}(\lambda) \;=\; \frac{1}{|W|} \sum_{j=1}^{|W|}\int_{\lambda-\lambda_j(k)=0} \frac{d\tilde S}{|\nabla_k \lambda_j|},
\end{equation}
in which $d\tilde S$ is the Euclidean volume of the Fermi surface $\Phi_\lambda\vcentcolon=\left\{ z\in(\CC^\times)^d : D(z,\lambda)=0 \right\}$ in $\TT^d$  at energy $\lambda$ (discussed below) scaled by $1/(2\pi)^d$.

For periodic operators, there is another formula for the density of states.  For a subset $\VV'\subset\VV$, let $\mathds{1}_{\VV'}$ denote the spatial projection (cutoff) of functions onto $\CC^{\VV'}$, and for a Borel set $B\subset\RR$, let $\chi_B$ denote the spectral projection onto $B$ in the spectral resolution of~$A$.  For any integer~$\ell$,
\begin{equation}
  \Lambda_\ell \;=\; [-\ell,\ell]^d \subset \ZZ^d,
\end{equation}
which is empty if $\ell<0$.  The density of states of $A$ is the Borel measure
\begin{equation}
  \delta(B) \;=\; \lim_{\ell\to\infty} \frac{\tr(\chi_B\mathds{1}_{W\dotplus\Lambda_\ell})}{\left|W\dotplus\Lambda_\ell\right|}.
\end{equation}
Since $A$ is periodic, $\chi_B$ commutes with the shift $S^n$, and thus $S^n\mathds{1}_{\VV'}S^{-n}=\mathds{1}_{\VV'\dotplus n}$, making $\chi_B\mathds{1}_{\VV'}$ unitarily equivalent to $\chi_B\mathds{1}_{\VV'\dotplus n}$.  Thus $\tr(\chi_B\mathds{1}_{W\dotplus\Lambda_\ell})=|\Lambda_\ell|\tr(\chi_B\mathds{1}_W)$, and we obtain
\begin{equation}
  \delta(B) \;=\; \frac{\tr(\chi_B\mathds{1}_W)}{|W|}.
\end{equation}
This is the von Neumann formula for density of states of a periodic operator.
A homological formula based on a free resolution of modules is derived in~\cite{Kravaris2023a}.

\subsection{Eigenfunctions: Module {\itshape vs.}\ $L^2$}\label{efcns}

Recall the dichotomy mentioned at the beginning of \S\ref{sec:periodic}.  This section shows how $A$ as a module endomorphism, that is, $A$ restricted to $F_0(\VV)$, gives full information on the $L^2$ eigenspaces of~$A$.

\subsubsection{Flat band functions}  
\S\ref{sec:eigenvalues} showed that eigenvalues of $\hat A(z)$ on $\CC[z^\pm]^W$ are of infinite multiplicity. 
Furthermore, for each $\zeta\in(\CC^\times)^d$, the sum of shifts $\sum_{n\in\ZZ^d}f(\bullet\dotminus n)\zeta^n$ (this is the Floquet transform at $z=\zeta$) is a $\lambda$-eigenfunction of $A$ and is $\zeta$-quasi-periodic.  Particularly, $\lambda$ is an eigenvalue of $\hat A(\zeta)$ for each $\zeta\in\TT^d$.  This means that $A$ has a constant, or ``flat", spectral band function at~$\lambda$.

\subsubsection{Module eigenfunctions and $L^2$ eigenfunctions}

In fact, the eigenvalues of $\hat A(z)$ on $\CC[z^\pm]^W$ are exactly all the eigenvalues of $A$ on~$\ell^2(\VV)$.

\begin{theorem}\label{thm:flat}
The following statements are equivalent.
\begin{enumerate}
  \item $\lambda_0$ is an eigenvalue of $\hat A(z)$, that is, there exists nonzero $f(z)\!\in\!\CC[z^\pm]$ such that $(\hat A(z)\!-\!\lambda_0)f(z)=0$ in~$\CC[z^\pm]$.
  \item $\lambda_0$ is an $L^2$ eigenvalue for $A$, that is, there exists $f\in L^2(\VV)$ such that $(A-\lambda_0)f=0$.
  \item $(\lambda-\lambda_0)$ is a factor of $D(z,\lambda)$.
\end{enumerate}
When these statements hold, $\lambda$ is an eigenvalue of $A$ of infinite multiplicity and the $L^2$ eigenspace contains a dense linear subspace of compactly supported eigenfunctions.  The set of eigenvalues of $A$ is finite. 
\end{theorem}

We prove (1)$\Rightarrow$(3)$\Rightarrow$(2)$\Rightarrow$(1), although (1)$\Rightarrow$(2) is straightforward. 

Statement (1) means that there exists a nonzero Laurent polynomial vector $f(z)\in\CC[z^\pm]^W$ such that $(\hat A(z)-\lambda_0 I)f(z)=0$ in $\CC[z^\pm]^W$.  Let $X$ be the variety on which $f(z)$ vanishes,
\begin{equation}
  X \;=\; \{ \zeta\in(\CC^\times)^d : f_v(\zeta)=0\; \forall v\in W \}.
\end{equation}
For all $\zeta\in(\CC^\times)^d\setminus X$, $f(\zeta)$ is a nonzero vector in $\CC^W$ and $(\hat A(\zeta)-\lambda_0 I)f(\zeta)=0$, so $D(\zeta,\lambda_0)=0$ as $(\CC^\times)^d\smallsetminus X$ is an open dense set.  Therefore $D(z,\lambda_0)=0$ as a Laurent polynomial in~$z$.  This means that $D(z,\lambda)$ has $(\lambda-\lambda_0)$ as a factor.  

Assuming (3), for each $\zeta\in(\CC^\times)^d$ there is a vector $f(\zeta)\in\CC^W$ such that $(\hat A(\zeta)-\lambda_0)f(\zeta)=0$, and the function $f(\zeta)$ can be taken to be continuous on $\TT^d$.
Fourier inversion provides a function $\check f\in \ell^2(\ZZ^d)$ such that $(A-\lambda_0)\check f=0$.

That (2) implies (1) follows from the remarkable fact that the space of compactly supported eigenfunctions is dense in the whole $\lambda$-eigenspace, which is proved below.  To see that the dimension of the eigenspace is infinite, let $f$ be a nonzero compactly supported $\lambda$-eigenfunction, 
note that, for large enough $s$, all the shifts of $f$ by elements of $s\ZZ^d$ have mutually disjoint supports and are therefore orthogonal.

\subsubsection{Approximation by compactly supported eigenfunctions}
First, let $\EE$ be a nontrivial shift-invariant subspace of the $\lambda$-eigenspace of $A$.  We will prove that there exists a compactly supported $\lambda$-eigenfunction of $A$ that is not orthogonal to $\EE$, a theorem due to Kuchment~\cite{Kuchment1989a}.  The proof is based on those of~\cite{Kuchment1989a,HiguchiNomura2009a}.  A proof for general amenable groups is in~\cite{Kravaris2023a}.

Denote by $P$ the orthogonal projection onto $\EE$ in~$\ell^2(\VV)$.  For any positive integers $\ell$ and $r$, set $\Lambda\!=\!\Lambda_\ell$ and $\partial=\Lambda_\ell\setminus\Lambda_{\ell-r}$.
The argument in \S\ref{DoS} yields
\begin{equation}\label{scaling}
  \frac{\tr(\mathds{1}_W P)}{|W|} \;=\; \frac{\tr(\mathds{1}_{W\dotplus\Lambda} P)}{|W||\Lambda|}.
\end{equation}
The space $\EE$ being nontrivial implies that $\tr(\mathds{1}_WP)>0$.  Since $\dim(\mathrm{im}\mathds{1}_{W+\partial})=|W\dotplus\partial|=|W||\partial|$, we obtain
\begin{equation}\label{one}
  \rk(\mathds{1}_{W\dotplus\partial}\EE) \leq |W||\partial|.
\end{equation}
Equation (\ref{scaling}) together with $\|\mathds{1}_{W\dotplus\partial}P\|\leq1$ yield
\begin{equation}\label{two}
  |\Lambda| \tr(\mathds{1}_WP) \;=\; \tr(\mathds{1}_{W\dotplus\Lambda}P) \;\leq\; \rk(\mathds{1}_{W\dotplus\Lambda}P).
\end{equation}
Fix $r$ fixed, observe that $|\partial|/|\Lambda|\to0$ as $\ell\to\infty$.  Since $\tr(\mathds{1}_WP)>0$, one can therefore choose $\ell$ such that
\begin{equation}\label{three}
  |\partial| \;<\; |\Lambda| \frac{\tr(\mathds{1}_WP)}{|W|}.
\end{equation}
Putting~\eqref{one},~\eqref{two}, and~\eqref{three} together yields
\begin{equation}
  \rk(\mathds{1}_{W\dotplus\partial}P) \;<\; \rk(\mathds{1}_{W\dotplus\Lambda}P).
\end{equation}
Consequently, there exists $\phi\in\EE$ such that $\phi\not\in\ker(\mathds{1}_{W\dotplus\Lambda})$ and $\phi\in\ker(\mathds{1}_{W\dotplus\partial})$.

Now apply this result with $r\!=\!2r(A)$ (see page~\pageref{Arange}).  Then $\phi$ is a $\lambda$-eigenfunction of $A$, $\phi$ is not identically zero on $\Lambda$, and $\phi$ vanishes on $\partial$.  It follows from the definition of $r(A)$ that the compactly supported function $\psi=\mathds{1}_{W\dotplus\Lambda}\phi$ is also a $\lambda$-eigenfunction of~$A$.  Since $(\psi,\phi)=(\psi,\psi)\not=0$, $\psi$ is not orthogonal to~$\EE$.

Denote by $\mathcal{F}$ the closure of the span of all compactly supported $\lambda$-eigenfunctions of~$A$, and let $\EE$ be its orthogonal complement within the $\lambda$-eigenspace of~$A$.  Since both $\mathcal{F}$ and the $\lambda$-eigenspace are shift-invariant, so is~$\EE$.  If $\EE$ is nontrivial, we have proved above that there is a compactly supported eigenfunction that is not orthogonal to $\EE$, which contradicts the definition of~$\EE$.  We conclude that $\mathcal{F}$ is equal to the $\lambda$-eigenspace of~$A$ and that therefore this eigenspace contains the compactly supported $\lambda$-eigenfunctions as a dense subset.

\section{Bloch and Fermi varieties}\label{sec:BlochFermi}

The relation $D(z,\lambda)=0$ is  the dispersion relation for the periodic operator $A$.
It has two fundamental meanings.  One is that if $D(z,\lambda)=0$ for some $z\in\TT^d$, then $\lambda\in\sigma(A)$.
The other is that, for any $\lambda\in\CC$ and $z\in(\CC^\times)^d$, $D(z,\lambda)=0$ implies that $A$ admits a Floquet mode at energy $\lambda$ and weight~$z$.
It defines the Bloch variety whose cross section at a fixed value of $\lambda$ is the Fermi variety at energy~$\lambda$.

\subsection{Definitions}

\subsubsection{The Bloch variety}

The dispersion relation of a periodic graph operator is the zero set of the dispersion function, and it is also known as
the {\itshape Bloch variety}.
The complex Bloch variety is
\begin{equation}
  \BB_{\!A} \;=\; \left\{ (z,\lambda)\in(\CC^\times)^d\times\CC : D(z,\lambda)=0 \right\},
\end{equation}
and the real Bloch variety is
\begin{equation}
  \RR\BB_{\!A} \;=\; \left\{ (z,\lambda)\in\TT^d\times\RR : D(z,\lambda)=0 \right\}.
\end{equation}
Below are three pictures of real Bloch varieties for the hexagonal lattice~\eqref{Eq:hexagonal_lattice}.
\begin{equation}\label{Eq:BlochVars}
  \raisebox{-52pt}{\begin{picture}(110,108)(0,-10)
      \put(0,0){\includegraphics[height=90pt]{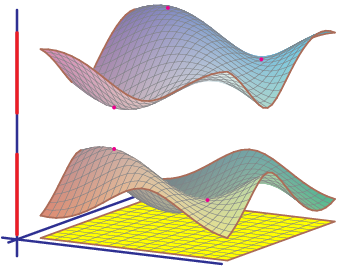}}
      \put(10,-10){$(a,b,c)=(-1,-3,-1)$}
    \end{picture}
    \qquad
 \begin{picture}(112,108)(0,-10)
      \put(0,2){\includegraphics[height=100pt]{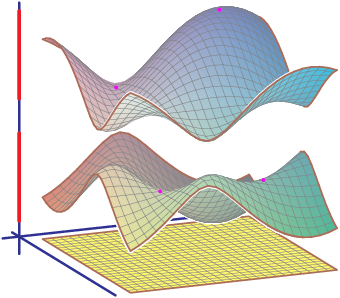}}
      \put(10,-10){$(a,b,c)=(-1,-1,-1)$}
    \end{picture}
    \qquad
   \begin{picture}(112,108)(0,-10)
      \put(0,2){\includegraphics[height=100pt]{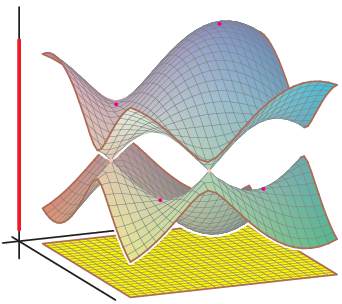}}
      \put(10,-10){$(a,b,c)=(-1,-1,-1)$}
    \end{picture}}
\end{equation}
The third has $V(v)=V(w)$ and is the Bloch variety of the graph Laplacian, later referred to as {\itshape graphene}.
Its two singular points are called Dirac points, and Figure~\ref{F:AA_and_AB} shows some perturbations of it.

\subsubsection{The Fermi variety}

The Fermi variety, or Fermi surface, is a cross section of the Bloch variety at a specific energy~$\lambda$.  For $\lambda\in\CC$, the complex Fermi surface~is
\begin{equation}
  \Phi_{\!A,\lambda} \;=\; \left\{ z\in(\CC^\times)^d : D(z,\lambda)=0 \right\},
\end{equation}
and for $\lambda\in\RR$, the real Fermi surface~is
\begin{equation}
  \RR\Phi_{\!A,\lambda} \;=\; \left\{ z\in\TT^d : D(z,\lambda)=0 \right\}.
\end{equation}

We show three real Fermi surfaces, which are curves in $\TT^2$, and a Bloch variety.
The first two Fermi surfaces are from the middle Bloch variety in~\eqref{Eq:BlochVars} at $\lambda=1.5$ and $\lambda=1.7$,
where the potential has been normalized so that $V(v)=0$ and $V(w)=1$, so that the three saddle points on the upper branch
are at $\lambda=\frac{1}{2}(1+\sqrt{5})$.
The third is from the Bloch variety at left; it is a slice through the middle branch (this was studied
in~\cite[Ex.\ 5.5]{FaustSottile2023a}).  
They are displayed in the fundamental domain $[-\frac{\pi}{2},\frac{3\pi}{2}]^2$.
\begin{equation}\label{Eq:Fermi}
  \raisebox{-35pt}{%
  \includegraphics[height=80pt]{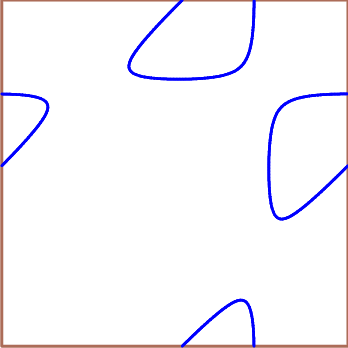}\qquad
  \includegraphics[height=80pt]{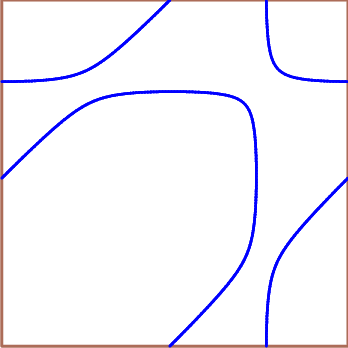}\qquad
  \includegraphics[height=80pt]{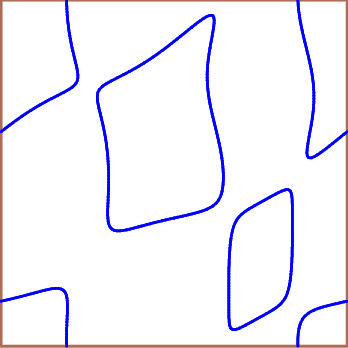}}\qquad
  \raisebox{-55pt}{\includegraphics[height=120pt]{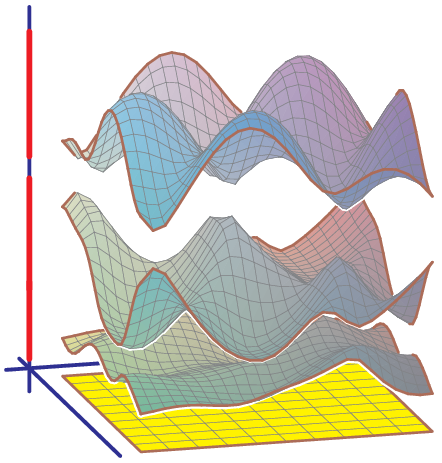}}
\end{equation}

\subsection{Spectrum of $A$}

The spectrum of a discrete periodic operator is determined by its Bloch variety through its projection onto the real $\lambda$ axis.  The spectrum is intimately connected with the Floquet modes of the operator.  Recall from \S\ref{sec:invspec} that the spectrum of $A$ is the the closed set
\begin{equation}
  \sigma_A \;=\; \left\{ \lambda\in\CC : A-\lambda I \;\text{does not have a bounded inverse in}\; \ell^2(\VV) \right\}.
\end{equation}
%

\subsubsection{Floquet modes}\label{sec:FloquetModes}
Floquet modes were introduced in \S\ref{sec:FB} as eigenfunctions of the shift group~$\ZZ^d$.
Floquet (or Floquet-Bloch) modes for a periodic operator $A$ are simultaneous eigenfunctions of $\ZZ^d$ and~$A$.  For each point $(z,\lambda)$ on the complex Bloch variety, we have a simultaneous eigenfunction of the shift group $\ZZ^d$ and the operator $A$.
That is, a function $f:\VV\to\CC$ such that
\begin{equation}\label{Floquet}
  n\cd f \;=\; z^n f 
  \qquad\text{and}\qquad
  A f \;=\; \lambda f.
\end{equation}
Such a function is called a {\itshape Floquet mode} for energy $\lambda$ and weight $z$, or quasi-momentum $k$, where $z=e^{i k}$.  Floquet modes are not in $\ell^2(\VV)$ because of the first equation.  The following well known result is fundamental.

\begin{theorem}\label{thm:Floquet}
  Let $\zeta\in(\CC^\times)^d$ and $\lambda\in\CC$ be such that $D(\zeta,\lambda)=0$.  Then there exists a Floquet mode for energy~$\lambda$ and weight~$\zeta$.
\end{theorem}

\begin{proof}
  Since $D(\zeta,\lambda)=0$, there exists $g\in\CC^W$ such that $(\hat A(\zeta)-\lambda I_W)g=0$.
  This makes the right-hand side of~\eqref{Eq:thmQA} vanish, with $\hat A(z)-\lambda I_W$ in place of $\hat A(z)$.
  Therefore $(\hat A(z)-\lambda I_W)[{\cal F} Q(g,\zeta)]=0$, so that
\begin{equation}
  (A-\lambda I) Q(g,\zeta) \;=\; 0.
\end{equation}
For $f=Q(g,\zeta)$, the second equation of (\ref{Floquet}) is satisfied, and the first is satisfied by definition of~$Q$.
\end{proof}

\subsubsection{Spectrum and $D(z,\lambda)$}\label{sec:spectrum}

The spectrum of a periodic operator $A$ has a geometric description:
\begin{equation}
  \sigma(A) \;=\; \left\{ \lambda\in\CC : \Phi_{\!A,\lambda} \cap \TT^d \not= \emptyset \right\}.
\end{equation}

\begin{theorem}
  A number $\lambda\in\CC$ is in the spectrum of $A$ if and only if there exists $\zeta\in\TT^d$ such that $D(\zeta,\lambda)=0$.
\end{theorem}

Fix $\lambda\in\CC$ and suppose $D(\zeta,\lambda)$ is nonzero for all $\zeta\in\TT^d$.  Then $(\hat A(\zeta)-\lambda)^{-1}$ exists as an analytic matrix-valued function on~$\TT^d$, and therefore is a bounded multiplication operator on~$L^2(\TT^d)$.  By the unitarity of the Fourier transform $\UU:\ell^2(\VV)\to L^2(\TT^d)$, the operator $A-\lambda$ is also invertible.  Now suppose that there exists $\zeta\in\TT^d$ such that $D(\zeta,\lambda)=0$.  Then by Theorem~\ref{thm:Floquet}, there exists a Floquet mode $f:\VV\to\CC$ for energy $\lambda$ and weight~$\zeta$.  For each positive integer $\ell$, define the compactly supported function $f_\ell:\VV\to\CC$~by
\begin{equation}
  f_\ell(x) \;=\; \frac{1}{\|f\|_W |\Lambda_\ell|^{1/2}} \mathds{1}_{W+\Lambda_\ell}(x)f(x),
\end{equation}
in which $\|f\|_W^2=\sum_{x\in W}|f(x)|^2$.  Set $r=2r(A)$ and $\partial_\ell = \Lambda_{\ell+r}\!\setminus\!\Lambda_{\ell-r}$.  For all $x\in\VV$,
\begin{equation}
  \| [(A-\lambda)f_\ell](x) \| \leq \frac{\|A-\lambda\|}{|\Lambda_\ell|^{1/2}},
\end{equation}
If $x\in\VV\!\setminus\!(W+\partial_\ell)$, then $[(A-\lambda)f_\ell](x)=0$.  Therefore,
\begin{equation}
  \|(A-\lambda)f_\ell\|_{\ell^2(\VV)} \;\leq\; \frac{\|A-\lambda\|\sqrt{|W||\partial_\ell|}}{|\Lambda_\ell|^{1/2}}
  \;\longrightarrow 0 \qquad (\ell\to\infty),
\end{equation}
whereas $\|f_\ell\|_{\ell^2(\VV)}\!=\!1$.  This implies that $A-\lambda$ does not have a bounded inverse on~$\ell^2(\VV)$.  The sequence $\{f_\ell\}_{\ell\in\mathbb{N}}$ is known as a Weyl sequence for~$A$.

\subsubsection{Self-adjointness and real structure}

When the operator $A$ is self-adjoint, the real Bloch and Fermi varieties are naturally real algebraic varieties, in a non-standard way.
Indeed, complex conjugation $z\mapsto \overline{z}$ is an involution on $\CC^n$ that is anti-holomorphic in that the induced map on
complex tangent spaces is conjugate-linear.
Its fixed points are the elements of~$\RR^n$.
More generally, a real structure on a complex space $X$ is an anti-holomorphic involution on $X$.
Its set of fixed points form the real points $X(\RR)$ of this real structure on $X$.
When they are nonempty and include a smooth point of $X$, we have that $\dim_\RR X(\RR) = \dim_\CC X$.

Recall from \S\ref{sec:invspec} that when $A$ is self-adjoint, $\hat A(z)$ satisfies the reflection principle
$\hat A(\bar z^{-1})\!=\!\hat A(z)^*$,
which implies the corresponding identity for~$D(z,\lambda)$,
\begin{equation}\label{Eq:D_is_reciprocal}
  D(\bar z^{-1},\bar\lambda) \;=\; \overline{D(z,\lambda)}, 
\end{equation}
for $z\in(\CC^\times)^d$ and $\lambda\in\CC$.

The associated involution $(z,\lambda) \mapsto (\bar z^{-1},\bar\lambda)$ is a non-standard real structure on the
space $(\CC^\times)^d\times\CC$ with real points $\TT^d\times\RR$.
It is non-standard as it differs from the usual structure $(z,\lambda)\mapsto(\bar z,\bar\lambda)$.
By~\eqref{Eq:D_is_reciprocal}, this restricts to a real structure on the complex Bloch variety $\BB_{\!A}$ and for
$\lambda\in\RR$, on the 
complex Fermi variety $\Phi_{A,\lambda}$.
Their real points are the real Bloch variety  $\RR\BB_{\!A}$ and real Fermi variety, respectively.

Thus the real Bloch variety has dimension $d$ in $\TT^d\times \RR$, and not $d{-}1$, which is what one expects from a
na\"ive dimension count: 
The complex Bloch variety has real codimension 2 in its ambient $(\CC^\times)^d\times\CC$, implying that the real Bloch
variety should have 
codimension 2 in $\TT^d\times\RR$, and hence dimension $d{-}1$.
Similarly, a real Fermi variety has dimension $d{-}1$ and not the na\"ive dimension count of $d{-}2$.
Moreover, the real Bloch variety determines and is determined by the complex Bloch variety, together with its non-standard
real structure.

There are further constraints imposed by $A$ being self-adjoint.
The property~\eqref{Eq:D_is_reciprocal} implies that the terms of the polynomial  $D(z,\lambda)$ come in pairs of the
form $a z^n\lambda^m+\bar a z^{-n}\lambda^m$ for some $a\in\CC$, $n\in\ZZ^d$, and $m\in\NN$.
Thus, for fixed $\zeta\in\TT^d$, $D(\zeta,\lambda)$ is a polynomial in $\lambda$ with real coefficients.
This only implies that set of roots of $D(\zeta,\lambda)=0$ are stable under complex conjugation.
In fact, as $A(\zeta)$ is hermitian, all these roots are real.

\subsubsection{Floquet sheaves and modules}

Floquet modes of~\S\ref{sec:FloquetModes} have an algebro-geometric interpretation,
The projection map $\scrE\vcentcolon=\bigl( (\CC^\times)^d\times\CC\bigr) \times \CC^W\to  (\CC^\times)^d\times\CC$ exhibits $\scrE$
as a (trivial) vector bundle with fiber $\CC^W$.
The characteristic matrix $\hat{A}(z)-\lambda I_W$ is an endomorphism of $\scrE$.
When $D(z,\lambda)\neq 0$, it is an isomorphism of the fiber $\CC^W$ of $\scrE$ at the point $(z,\lambda)$,
and when $D(z,\lambda)= 0$,  Theorem~\ref{thm:Floquet} implies that the  Floquet modes at $(z,\lambda)$ are exactly the
kernel of  $\hat{A}(z)-\lambda I_W$.
In the language of algebraic geometry~\cite{ShafII}, the Floquet modes form the kernel sheaf $\scrF\subset \scrE$ of the endomorphism
$\hat{A}(z)-\lambda I_W$ of $\scrE$, and the support of this Floquet sheaf $\scrF$ is the Bloch variety.
In the examples we give, the fibers of $\scrF$ at smooth points of $\BB_{\!A}$ are 1-dimensional, and at the singular (Dirac) points
in~\eqref{Eq:BlochVars}, they have dimension 2.

The matrix $\hat{A}(z)-\lambda I_W$ of polynomials in $\CC[z^\pm,\lambda]$ is also an endomorphism of the free  
$\CC[z^\pm,\lambda]$-module  $\CC[z^\pm,\lambda]^W$.
The counterpart of the Floquet sheaf $\scrF$ is the kernel of this map  $\hat{A}(z)-\lambda I_W$, which is also a module for
the ring $\CC[z^\pm,\lambda]$ (for more on modules over commutative rings, see~\cite[Ch.~5]{CLOII}).
By the dictionary between affine algebraic varieties and finitely generated commutative $\CC$-algebras, these two perspectives,
kernel sheaf and kernel module, are equivalent ways to view the same object.
This module-theoretic point of view is due to Kravaris~\cite{Kravaris2023a}, who reinterpreted work of Kuchment~\cite{Kuchment1989a}.
Kravaris gave new results on the density of states using free resolutions for general periodic operators, generalizing
earlier work of Gieseker, Kn\"{o}rrer, and Trubowitz~\cite{GiesekerKnorrerTrubowitz1993} for the square lattice $\ZZ^2$.

\subsection{Reducibility and defect eigenvalues}

The reducibility of the Bloch and Fermi varieties has spectral consequences, notably concerning the existence of embedded eigenvalues caused by local defects, as discovered by Kuchment and Vainberg~\cite{KuchmentVainberg1998,KuchmentVainberg2000,KuchmentVainberg2006}.  Since a periodic operator is well defined without reference to a specific fundamental domain $W$, it stands to reason that reducibility is independent of the choice of fundamental domain $W$.  We omit an analysis of this point.

\subsubsection{Components of the Bloch variety}
As an element of $\CC[z^\pm,\lambda]$, $D(z,\lambda)$ has a decomposition into irreducible elements that are unique up to multiplication by constants and monomials in $\CC[z^\pm]$,
\begin{equation}\label{Bfactorization}
  D(z,\lambda) \;=\; F(\lambda) \prod_{j=1}^\alpha \tilde D_j(z,\lambda).
\end{equation}
All the factors that are polynomials in $\lambda$ alone are incorporated into 
  $F(\lambda) = \prod_{\ell=1}^\gamma (\lambda-\lambda_\ell)^{p_\ell}$,
with the real numbers $\lambda_\ell$ being distinct.  We do not require the $\tilde D_j$ to be distinct.
Each factor $\tilde D_j(z,\lambda)$ corresponds to a component of the Bloch variety
$\BB_{\!A}$,
\begin{equation}
  \BB_{\!\!A}^j \;=\; \left\{ (z,\lambda) \in(\CC^\times)^d\times\CC : \tilde D_j(z,\lambda)=0 \right\}.
\end{equation}
Each of these components contributes a set of spectral bands and gaps to the spectrum of~$A$,
\begin{equation}
  \tilde\sigma_j(A) = \Big\{ \lambda\in\RR : \exists z\in\TT^d,\,(z,\lambda)\in \BB_{\!\!A}^j \Big\},
\end{equation}
and each number $\lambda_\ell$ is an eigenvalue of~$A$ of infinite multiplicity.  Thus,
\begin{equation*}
     \sigma(A) \;=\; \{\lambda_\ell\}_{\ell=1}^\gamma \cup\bigcup_{j=1}^\alpha \tilde\sigma_j(A).
\end{equation*}
%

Roots $\lambda_\ell$ of $F(\lambda)$ are flat bands of the Bloch variety.
The Bloch variety of the  Lieb lattice (on the left below) has a flat band when the potentials at the vertices of
degree 2 are equal.
 \begin{equation}\label{Eq:Lieb}
  \raisebox{-40pt}{\includegraphics[height=95pt]{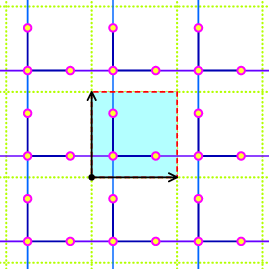}\qquad
  \includegraphics[height=95pt]{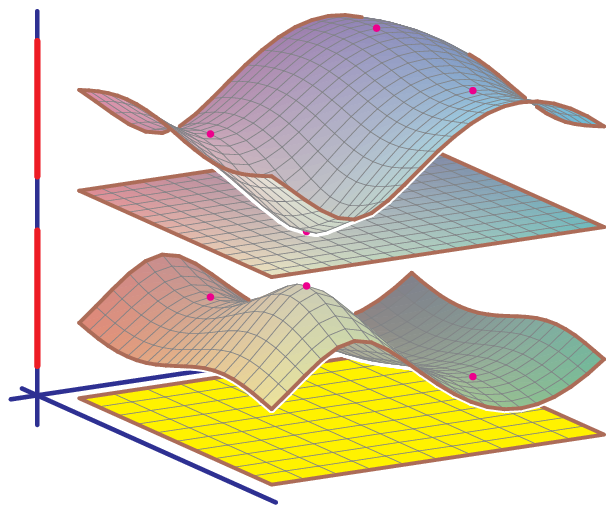}}
 \end{equation}
 This observation is made more precise in recent work.
 The set of parameters (potential and edge weights) whose corresponding Bloch variety has a flat band is not dense in the
 set of parameters~\cite{faust2025rareflatbandsperiodic}.
 This is strengthened in~\cite{FaustKachkovskiy}, where it is shown that for fixed edge weights, the set of potentials whose
 whose corresponding Bloch variety has a flat band is not dense.

\subsubsection{Components of the Fermi variety}  Let $\lambda_0\!\in\!\CC$ be fixed.  The factorization of the $z$-dependent factors of $D(z,\lambda)$ in (\ref{Bfactorization}) persists for $D(z,\lambda_0)$ as an element of $\CC[z^\pm]$, but it may possibly be refined further, as each factor $\tilde D_j(z,\lambda_0)$ has its own factorization in $\CC[z^\pm]$,
\begin{equation}\label{Ffactorization}
  D(z,\lambda_0) \;=\; F(\lambda_0) \prod_{\ell\in[1,\beta]} D_\ell(z,\lambda_0),
\end{equation}
with $\beta\geq\alpha$.
Each $\tilde D_j(z,\lambda_0)$ is a product of some of the~$ D_\ell(z,\lambda_0)$.
Note that the number $\beta$ of factors is taken to be independent of~$\lambda_0$, and that this factorization is
meaningful even if $F(\lambda_0)\!=\!0$. One can prove that either $D(z,\lambda_0)$ is reducible for all
$\lambda_0\!\in\!\CC$ or it is reducible for only a finite set of values of $\lambda_0$ (this is known as Hilbert's
Irreducibility Theorem). For example, the Bloch variety of the pure Schr\"odinger operator on the graph in Figure~\ref{F:Dirac_4} is
irreducible, but its Fermi variety at the energy of its Dirac points is highly reducible and is situated on the Bloch
variety in an interesting way.
It 
\begin{figure}[htb]
 \centering
     \includegraphics[height=95pt]{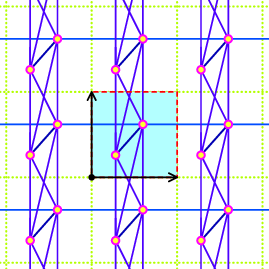}\qquad
     \raisebox{-15pt}{\includegraphics[height=125pt]{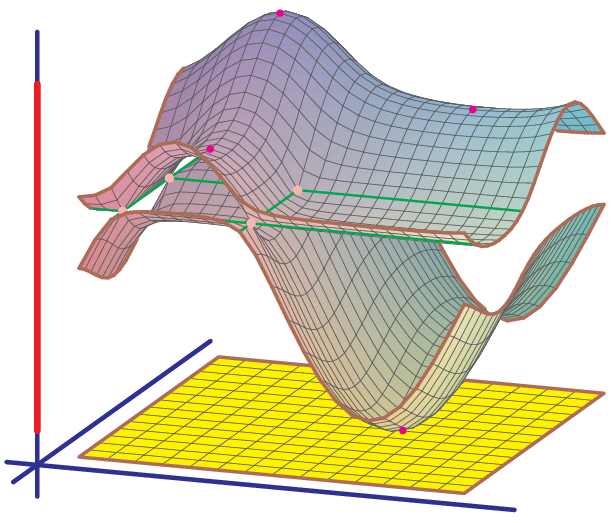}}
       \caption{Graph and Bloch variety with highly reducible Fermi curve}
       \label{F:Dirac_4}
\end{figure}
consists of two pairs of parallel lines (the green lines in Figure~\ref{F:Dirac_4}) that meet at the four Dirac points.
Each line lies on each branch of the Bloch variety.

Each factor $D_j(z,\lambda)$ corresponds to a component of the Fermi variety~$\Phi_{\!A,\lambda_0}$,
\begin{equation}
  \Phi_{\!\!A,\lambda_0}^j \;=\; \left\{ z\in(\CC^\times)^d : D_j(z,\lambda_0)=0 \right\}.
\end{equation}
Each of these components contributes a set of spectral bands and gaps to the spectrum of~$A$, which refine the bands~$\tilde\sigma_j(A)$,
\begin{equation}
  \sigma_j(A) = \Big\{ \lambda_0\in\RR : \TT^d\cap\Phi_{\!\!A,\lambda_0}^j \not= \emptyset \Big\}.
\end{equation}

\subsubsection{$L^2$ response to source in the continuum}\label{sec:response}
We may view $\hat A(z)$ as a self-adjoint analytic family over $\TT^d$ of operators on~$\CC^W$.  Fix $\lambda\in\RR$, and consider the equation
\begin{equation}\label{feqn}
  (A-\lambda)u \;=\; f
  \qquad \text{in\, $\ell^2(\VV)$}.
\end{equation}
Under the Fourier transform, this becomes  $(\hat A(z)-\lambda I)\hat u(z)\!=\!\hat f(z)$ on $L^2(\TT^d,\CC^W)$, and hence
\begin{equation}\label{fsoln}
  \hat u(z) \;=\; \frac{R(z,\lambda)\hat f(z)}{D(z,\lambda)},
\end{equation}
in which $R(z,\lambda)$ is the adjugate of $(\hat A(z)-\lambda I)$, a Laurent polynomial in $z$ with coefficients in~$\CC^{W\times W}$.  Note that $u\in\ell^2(\VV)$ if and only if $\hat u \in L^2(\TT^d,\CC^W)$, and this happens when the right-hand side of (\ref{fsoln}) is regular on~$\TT^d$, that is, the numerator cancels the zero set of the denominator.

Let $D(z,\lambda)=\Dg(z,\lambda)\Ds(z,\lambda)$, where $\Dg(z,\lambda)$ is the product of all factors $D_j$ of $D$ whose zero sets do not intersect $\TT^d$ and $\Ds(z,\lambda)$ has the rest of the factors in (\ref{Ffactorization}).  Note that the real variety of $\Dg$ in $\TT^d$ is empty.  The subscripts refer to ``gap" and ``spectrum".  The key observation is that, in order for for $R(z,\lambda)\hat f(z)$ to cancel the zero set of $\Ds(z,\lambda)$ on $\TT^d$, it must cancel $\Ds(z,\lambda)$ identically as a Laurent polynomial in~$z$.
The proof of the following theorem is contained in the proof of~\cite[Theorem~6]{KuchmentVainberg2006}.

\begin{theorem}
  In order for $R(z,\lambda)\hat f(z)/D(z,\lambda)$ to be regular on $\TT^d$, each vector component of $R(z,\lambda)\hat f(z)$ must have $\Ds(z,\lambda)$ as a factor.
\end{theorem}

Na\"ively, this factorization is accomplished if $\hat f(z)\!=\!\Ds(z,\lambda)\hat f_0(z)$ for some $f_0(z)\in\CC[z^\pm]^W$.  But additional structure may allow $\hat f(z)$ to be chosen so that $R(z,\lambda)\hat f(z)$ has $\Ds$ as a factor even if $\hat f(z)$ does not.

\subsubsection{Reducibility of $\BB_{\!A}$ by symmetry}\label{sec:symmetry}

This section is a generalization of a construction in~\cite{Shipman2014}.  A~finite symmetry of $A$ that commutes with $\ZZ^d$ yields reducibility of the Bloch variety.
Denote a free action $G\times W\to W$ of a finite group $G$ on $W$ by $(g,x)\mapsto g\cd x$.  Thus $\ZZ^d\times G$ acts on $\ZZ^d\times W$ and therefore also on $\ell^2(\VV)=\ell^2(\ZZ^d)\otimes\CC^W$~by
\begin{equation}
  [(n,g)f](m,y) \;=\; f(m\dotminus n,g^{-1}y),
\end{equation}
for all $m\in\ZZ^d,\,y\in W$ and $n\in\ZZ^d,\,g\in G$.  The vector space $\CC^W\!\!$, as a representation of~$G$, admits a decomposition
  $\CC^W = \bigoplus_{j\in[1,r]} X_j$,
in which $X_j$ is a direct sum of $s_j$ copies of one of the $m$ irreducible representations of~$G$.  This induces an orthogonal decomposition
\begin{equation}\label{symmdecomp}
  \ell^2(\VV) \;=\; \ell^2(\ZZ^d)\otimes\CC^W \;=\; \bigoplus_{j\in[1,r]}\!\! \ell^2(\ZZ^d)\otimes X_j.
\end{equation}

Now let $A$ commute also with the $G$ action.  Thus $A$ is invariant on each component of the decomposition (\ref{symmdecomp}) and may be further decomposable within each of the components $\ell^2(\ZZ^d)\otimes X_j$.  $A$ can be conceived as a graph operator as follows.  Let $W_j$ be an orthonormal basis for $X_j$ as a vector space over~$\CC$, and, for each $j\in[1,r]$, define a vertex set
\begin{equation}
  \VV_j \;=\; \ZZ^d\times W_j.
\end{equation}
The set $W_j$ serves as a fundamental domain of the $\ZZ^d$ action on $\VV_j$, and (\ref{symmdecomp}) becomes
\begin{equation}
  \ell^2(\VV) \;=\; \bigoplus_{j\in[1,r]} \ell^2(\VV_j).
\end{equation}
The operator $A$ is decomposed into its projections onto these subspaces, that is
  $A = \sum_{j\in[1,r]} A_j$,
where $A_j$ is the restriction of $A$ to~$\ell^2(\VV_j)$.  Under the Floquet transform, we obtain
  $\hat A(z) = \bigoplus \hat A_j(z)$,
and ultimately
\begin{equation}
  D(z,\lambda) \;=\; \prod_{j\in[1,r]}\! P_j(z,\lambda)^{s_j},
\end{equation}
in which $P_j(z,\lambda)^{s_j} := \det(\hat A_j(z)-\lambda)$.  Possible further factorizations of the $P_j$ yield the $\tilde D_j$ above.

\subsubsection{Generalized symmetries}\label{sec:gensym}
Begin with an operator $\mathring A$ on  $\ell^2(\mathring\VV)$, with dispersion function $\mathring D(z,\lambda)$.  Create a new vertex set consisting of $s$ copies (layers) of $\mathring\VV$, that is, $\VV=\sqcup_{j\in[1,s]}\mathring\VV$.  We will create a ``multilayer" operator $A$~on
\begin{equation}
  \ell^2(\VV) \;=\; \CC^s\otimes\ell^2(\mathring\VV).
\end{equation}
Let $\{P_j\}_{j\in[1,r]}$ ($r\leq s$) be orthogonal projectors onto subspaces of $\CC^s$.  These subspaces generalize the spaces $X_j$ above, as they are not necessarily coming from an underlying symmetry of~$W$.
Define $A$~by
\begin{equation}
  A \;=\; I_s \otimes \mathring A \,+\, \sum_{j\in[1,r]} P_j\otimes L_j \,=\, \sum_{j\in[1,r]} P_j\otimes (\mathring A+L_j),
\end{equation}
in which the $L_j$ are periodic graph operators on~$\mathring\VV$.  The image of the projector $P_j\otimes I$ is an invariant space of ``hybrid states" for $A$, and $A$ acts on this space by $A+L_j$, as if it were a ``single-layer" operator on~$\ell^2(\mathring\VV)$, possibly with multiplicity.  Now $D(z,\lambda)$ has factors $P_j(z,\lambda)=\det(\hat A(z)+\hat L_j(z)-\lambda I)$.

A simple case of this construction (see~\cite{Shipman2014}) occurs when all $m$ layers are coupled by one self-adjoint $m\times m$ matrix $K=\sum_{j\in[1,r]}\lambda_j P_j$ and $L_j=I$,
\begin{equation}\label{mcoupled}
  A \;=\; I_s\otimes \mathring A + K\otimes I  \;=\; \sum_{j\in[1,r]} P_j \otimes (\mathring A+\lambda_j I).
\end{equation}
The dispersion function is
  $D(z,\lambda) = \prod_{j\in[1,r]}\! \mathring D_j(z,\lambda-\lambda_j)^{s_j}$,
so $\sigma(A)$ is a union of shifts of $\sigma(\mathring A)$.  

\subsubsection{Example: AA-stacked bi-layer graphene}\label{sec:AA}
The simplest model of AA-stacked graphene is two copies of single-layer graphene $H_0$ coupled by a $2\times2$ interlayer hermitian matrix~$\Gamma$.  Let $\VV_0$ be the vertex set for the single layer and $D_0(z,\lambda)$ its dispersion function.   Then the AA-stacked model $H_\text{AA}$ acts in $\CC^2\otimes\ell^2(\VV_0)$,
\begin{equation}
  H_\text{AA} = I_2\otimes H_0 + \Gamma\otimes I.
\end{equation}
Let $\Gamma=U\Lambda U^{-1}$ with $U$ unitary and $\Lambda=\diag(\mu_1,\mu_2)$.  Then $\mathcal{U}=U\otimes I$ block-diagonalizes $H_\text{AA}$, with the blocks being spectrally shifted copies $H_0+\mu_j I$ of the single layer.  Thus the dispersion function for $H_\text{AA}$ is reducible,
\begin{equation}
  D(z,\lambda) = D_0(z,\lambda+\mu_1)D_0(z,\lambda+\mu_2).
\end{equation}
This may be seen in the Bloch variety on the right in Figure~\ref{F:AA_and_AB} (in two pages).

\subsubsection{Reducibility of $\Phi_{\!A,\lambda}$ by contracting edges}
The coupling of multiple layers in \S\ref{sec:gensym} can be generalized to more elaborate coupling graphs, which results in the shifts $\lambda_j$ in (\ref{mcoupled}) being $\lambda$-dependent.  This has been worked out for quantum graphs, where the edges are endowed with an ordinary differential operator in~\cite{BrownSchmidtShipmanWood2021,Shipman2019} and then for discrete graphs in~\cite{Villalobos2024a}.

Two decoupled layers of $\mathring A$ on $\ell^2(\mathring\VV)$ yield the operator $I_2\otimes\mathring A$ on $\CC^2\otimes\ell^2(\mathring\VV)$, with vertex set $\mathring\VV\!\sqcup\!\mathring\VV=\{1,2\}\!\times\!\mathring\VV$.  Now, for each pair of ``aligned" vertices, say $(1,v)$ and $(2,v)$, create a vertex set $\VV_v$ containing these two vertices, and a self-adjoint coupling operator $K_v$ on~$\VV_v$, ensuring periodicity $K_{v\dotplus n}=K_v$.  Then merge these connecting operators with $I_2\otimes\mathring A$ to obtain a periodic ``bilayer" operator with vertex set $\sqcup_{v\in\mathring\VV}\VV_v$.

Consider now the equation $(A-\lambda)u=f$, where $f$ is supported on the vertices $\mathring\VV\!\sqcup\!\mathring\VV$.  This equation can be written equivalently by eliminating the extra vertices in $\VV_v$, for each~$v\in\mathring\VV$, by replacing $K_v$ with its Schur complement.  This is a $2\!\times\!2$ matrix indexed by the vertices $(1,v)$ and $(2,v)$ and is a rational function of~$\lambda$.  Explicitly, write the matrix for $K_v\!-\!\lambda$ in block form, with $(1,v)$ and $(2,v)$ coming first in the ordering,
\begin{equation}
  K_v \;=\; \mat{1.1}{A_v-\lambda}{B_v}{B_v^*}{C_v-\lambda}.
\end{equation}
The Schur complement is $S_v = A_v-\lambda-B_v(C_v-\lambda)^{-1}B_v^*$.  If the $K_v$ are the same for all $v\in\mathring\VV$, the resulting contracted operator has the form of (\ref{mcoupled}) with $m\!=\!2$ and $S_v$ taking the place of~$K$.  More generally, if the $K_v$ all commute with each other, then they admit common spectral projections $P_1$ and $P_2$, but the eigenvalues depend on~$v$ periodically; denote them by $\mu_j^v$ for $j\in\{1,2\}$ and $v\in\mathring\VV$.  The resulting contracted operator is $\lambda$ dependent,
\begin{equation}
  \tilde A(\lambda) \;=\; \sum_{j\in[1,2]} P_j \otimes \big(\mathring A + \sum_{v\in\mathring\VV} \mu_j^v(\lambda) \big),
\end{equation}
in which $\mu_j^v(\lambda)$ is treated as a function of~$v$ and acts as an onsite potential.  It generalizes $\lambda_j$ in (\ref{mcoupled}), which is independent of $\lambda$ and $v$.

The Fermi dispersion function can be factored as $D(z,\lambda) \;=\; D_1(z,\lambda)D_2(z,\lambda)$, with $D_j(z,\lambda)=\det(\hat{\mathring A}(z) + \diag_v\{\mu_j^v(\lambda)\})$.  Since these factors are polynomial in $z$ and rational functions of $\lambda$, the factorization yields reducibility of the Fermi variety, for each $\lambda$ that is not a pole, but not for the Bloch variety.

This construction can be done essentially verbatim with $m$ layers.

\subsubsection{Reducibility by reduction to one variable}\label{sec:types}
For certain multi-layer graphs not possessing any symmetry, the Fermi surface may still be reducible.  Another mechanism for this occurs when the dispersion function is a polynomial $P(\xi,\lambda)$ in a single composite momentum variable $\xi\!=\!g(z)\in\CC$ and energy~$\lambda$, with $g(z)\in\CC[z^\pm]$.  This means that the Bloch variety factors through the Riemann surface $\mathcal{R}\!=\!\{(\xi,\lambda)\in\CC^2 : P(\xi,\lambda)=0\}$.  For $\lambda_0\!\in\!\CC$, let $\mathcal{R}_{\lambda_0}$ denote the finite set of roots $\xi$ of $P(\xi,\lambda_0)$; the factorization can be seen in the following diagram.
\begin{equation}
\renewcommand{\arraystretch}{}
\left.
  \begin{array}{ccccc}
  (\CC^\times)^d\times\CC & \overset{z\mapsto\xi}{\longrightarrow} &
  \CC^2 & \overset{P}{\longrightarrow} & \CC \\
  \BB & \longrightarrow & \mathcal{R} & \longrightarrow & 0 \\
  \Phi_{\lambda_0} & \longrightarrow & \mathcal{R}_{\lambda_0} & \longrightarrow & 0 \\
  \end{array}
\right.
\end{equation}
The Fermi surface at energy $\lambda_0$ is the zero set of
\begin{equation}
  P(\xi,\lambda_0) \;=\; \prod_j \big(\xi - \xi_j(\lambda_0)\big),
  \qquad \xi = g(z).
\end{equation}
In this way, the Fermi surface $\Phi_{\lambda_0}$ has components $\Phi^j_{\lambda_0}=\{z\in(\CC^\times)^d : g(z)=\xi_j(\lambda_0)\}$, which depend on $\lambda_0$ algebraically, through a multi-valued function on~$\mathcal{R}$.

Two types of multi-layer quantum graphs that are reducible by reduction to one variable are introduced in \cite{FisherLiShipman2021}, and the discrete graph version is treated in~\cite{Villalobos2024a}.  The first type relies on a calculus for computing the dispersion function for the joining of two periodic graphs by merging pairs of vertices, with one vertex per fundamental domain in each graph.  The second relies on the bipartiteness of the layers, each having $|W|=2$.  One deduces, for example, that very general stacking of multiple layers of graphene with arbitrary shifts, results in a Fermi surface with as many components as the number of layers.

\subsubsection{Example: AB-stacked bi-layer graphene}
Bi-layer graphene with one layer shifted relative to the other is called AB-stacked, or Bernal-stacked, graphene.  It lies within the class of both types discussed in \S\ref{sec:types}.  The Fermi surface for the Hamiltonian
\begin{equation}\label{AB}
 A(z) \;=\;
\renewcommand\arraystretch{1.0}
\left[\begin{array}{cc|cc}
\Delta & \zeta' & \gamma_4\zeta' & 0 \\ 
\zeta & \Delta & \gamma_1 & \gamma_4\zeta' \\\hline
\gamma_4\zeta & \gamma_1 & -\Delta & \zeta' \\
0  
& \gamma_4\zeta & \zeta & -\Delta
\end{array}
\right]\,.
\end{equation}
(written with respect to the basis (1A, 1B, 2A, 2B)) is reducible into two components, coming from the factorization $D(z,E)=D_1(z,E)D_2(z,E)$ in $\CC[z^\pm]$ for each~$E\in\CC$.  Each factor $D_\ell(z,E)$ is a polynomial in $\xi=\zeta\zeta'$, with $\zeta=1+z_1+z_2$ and $\zeta'=1+z_1^{-1}+z_2^{-1}$.  Physical implications of this reducibility are studied in~\cite{MassattShipmanVekhterWilson2025}.

We show
Bloch varieties for both AB- and AA- stacked bi-layer graphene.
They are perturbations of the doubling of the Bloch variety of graphene (on the right in~\eqref{Eq:BlochVars}).
%
 \begin{equation}\label{F:AA_and_AB}
  \raisebox{-45pt}{\includegraphics[height=100pt]{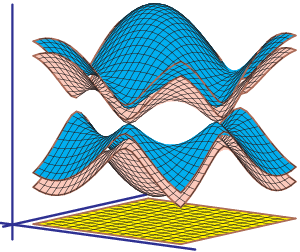}\qquad
  \includegraphics[height=120pt]{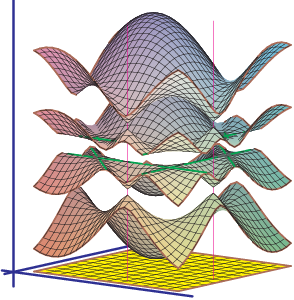}}\qquad
  \raisebox{-55pt}{\includegraphics[height=120pt]{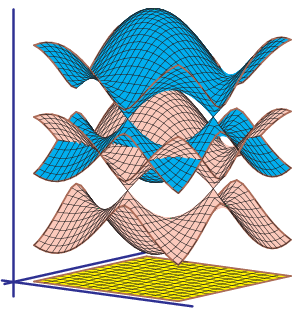}}
 \end{equation}
%
%
%
%
In the first, the Dirac points of graphene remain double points, but they are smoothed.
For this, the operator is~\eqref{AB} with $\Delta=\gamma_4=0$ and $\gamma_1=1/2$.
The second is~\eqref{AB} with $\Delta=1$, $\gamma_4=0$, and $\gamma_1=2/3$.
In it, the Dirac points are smoothed and slightly pulled apart (the faint vertical lines are to help show this).
Curiously, the Fermi curve at the endpoints of the band gap are reducible curves of critical points (see
\S\ref{S:SENDC}).  The third is AA stacked bi-layer graphene, and it consists of two shifted copies of the Bloch variety of graphene.

\subsubsection{Defect states in the continuum}
A consequence of reducibility of the Fermi surface is the ability to create a defect state at an energy in the continuous spectrum supported by a local defect, where the state is exponentially decaying and not compactly supported.  This is essentially a converse of \cite[Theorem 6]{KuchmentVainberg2006}.  The equation is
\begin{equation}\label{bseqn}
  (A+V-\lambda)u \;=\; 0,
\end{equation}
in which $V$ is a potential operator with finite support, $u\!\in\!\ell^2(\VV)$, and $\lambda\!\in\!\sigma(A)$.  

We point out that, when the operator has symmetry such as in \S\ref{sec:symmetry} and \S\ref{sec:AA}, one does not need to use the reducibility of the Fermi variety, or even the periodicity of the operator, to construct bound defect states in the continuum.  The decomposition of the operator onto orthogonal spaces of states is enough.  For examples, see~\cite{Papanicolaou2021a} for fourth-order ODEs and~\cite{Shipman2014,ShipmanVillalobos2023a} for periodic and quasi-periodic quantum graphs.

To create such a defect mode, first solve the local-source equation $(A-\lambda)u=f$ (\ref{feqn}).
Consider the example of AB-stacked graphene.  Recall from \S\ref{sec:response} that, to ensure $\lambda\!\in\!\sigma(A)$, $\lambda$ has to be chosen such that one of the factors $D_\ell(z,\lambda)$, say $D_2$, vanishes at some $z\in\TT^2$.  To satisfy $u\!\in\!\ell^2(\VV)$, the vector $R(z,\lambda)\hat f(z)$ must have $D_2(z,\lambda)$ as a factor in each component.  In order to satisfy the additional requirement that $u$ does not have compact support, $D_1$ must not vanish anywhere on $\TT^2$ and $R(z,\lambda)\hat f(z)$ cannot have $D_1$ as a factor in each component.  Thus $\hat u$ is a smooth rational function on $\TT^2$ so that $u$ is properly exponentially decaying.
Then put
\begin{equation}
  f=-Vu
\end{equation}
to obtain~(\ref{bseqn}).  Since $f$ has compact support, this amounts to a finite-dimensional system for~$V$.  If $u$ vanishes at some $x\in\VV$, then $V$ cannot be purely on-site; however generically one can find a local potential operator $V$ that works, as demonstrated in~\cite{FisherLiShipman2021,MassattShipmanVekhterWilson2025}.
Embedded eigenvalues created by non-compactly supported potentials require other techniques; see
\cite{JudgeNabokoWood2018,Liu2018} for results in one dimension.

\subsubsection{Localization of the defect}
Let us continue the AB-stacked graphene example.  As described at the end of \S\ref{sec:response}, 
having a factor of $D_2(z,E)$ in each component of the vector $R(z,\lambda)\hat f(z)$ can be achieved na\"ively if $\hat f(z)\!=\!D_2(z,\lambda)\hat f_0(z)$ for some $f_0(z)\in\CC[z^\pm]^W$.  This means that $f$ and therefore $V$ generically would be nonzero on 28 vertices ($D_2$ has seven monomial terms and there are four vertices in a fundamental domain~$W$).
However, it is desirable to seek $\hat f(z)$ with the minimal number of monomial terms in all components.  This corresponds to a source $f$ having minimal support.

For the Hamiltonian~\ref{AB}, it is shown in \cite{MassattShipmanVekhterWilson2025} there is a non-obvious matrix factorization $U(\hat A(z)-\lambda I)L = \Lambda$, with $U$ and $L$ being block upper and lower matrices in an appropriate basis, and $\Lambda$ being essentially a $2\times 2$ matrix with a special form.  Remarkably, the eigenvectors $\phi_\ell$ of $\Lambda$ are independent of $z$ and its eigenvalues happen to be~$D_\ell$.  Furthermore, $\phi_\ell$ and $U\phi_\ell$ have nonzero components only in the aligned pair of A and B vertices of~$W$.  Therefore, taking $\hat f=\phi_1$ results in $\hat u =(\hat A(z)-\lambda I)^{-1}\hat f=L\Lambda^{-1}U\hat f$ having poles on $D_1(z,\lambda)=0$, which by assumption does not intersect~$\TT^2$.  Since $\hat f$ is independent of $z$ and has two non-zero components, $f$ is supported on two vertices.  The details are in \cite{MassattShipmanVekhterWilson2025}, including a formula for the values of $V$ on the two aligned AB vertices.

\subsubsection{Irreducible Fermi varieties}
Although na\"ively the Fermi variety of a discrete periodic operator (with $d\geq2$) should be generically irreducible, proving irreducibility is not straightforward.  Compared with proving reducibility, the techniques are different; this is because reducibility arises constructively by creating graphs, typically through multiple layers, that are explicitly designed to be algebraically reducible.

Although there may be no crisp way to define when a periodic graph has multiple layers, for $d=2$ it is reasonable to consider a planar graph to be a single layer.  When $|W|=2$ and the graph is planar, reducibility of the Fermi variety occurs only for the tetrakis grid, and positivity of the matrix elements of the operator always prohibits reducibility.  This is proved by computational means~\cite{LiShipman2020}.

Liu~\cite{Liu2020} proved irreducibility of the discrete Laplacian plus a periodic potential at every energy (except one when $d=2$), from which is inferred generic reduction of the dimension of the Fermi variety at band edges.
More generally, let $A$ be a $\ZZ^d$ periodic operator on $\ZZ^d$ so that $\hat A(z)$ is scalar, and let $V:\ZZ^d\to\CC$ be a potential function that is periodic with respect to a sublattice $q_1\ZZ\times\dots\times q_d\ZZ$.  Irreducibility of the Bloch variety, but not necessarily the Fermi varieties, of $A+V$ occurs under a condition connecting $\hat A(z)$ and the periods~$q_j$~\cite{FillmanLiuMatos2022}.  
By bounding the number of irreducible components in terms of certain asymptotics of $D(z,\lambda)$, irreducibility can be established for a wider class of examples~\cite{FillmanLiuMatos2024a}.

\subsubsection{Isospectrality}
There are some results about the information that the Bloch and Fermi varieties contain.
For Schr\"odinger operators on $\ZZ^d$ equal to the discrete Laplacian plus a potential that is periodic with a rectangular fundamental domain,
the density of states determines the Bloch variety \cite{BattigKnorrerTrubowitz1991}.  
Liu introduces Fermi isospectrality~\cite{Liu2023a,Liu2024a} of two such potentials when the associated Schr\"odinger operators have the same Fermi variety for some energy~$\lambda$.
For $d\geq3$, if two potentials are Fermi-isospectral either both are additively separable or neither is.
In the former case, the one-variable summands are essentially Fermi-isospectral, which also holds for $d=2$.
There exist complex potentials that are Floquet-isospectral to the zero potential~\cite{FLMPRTTZ}.

\subsection{Changing the lattice}

We ask how the Fermi variety transforms when one changes the fundamental domain or considers only a subgroup of shifts.  Of course, $\sigma(A)$ is unaffected because it does not rely on these choices.

\subsubsection{Some algebra}
Let $\Lambda$ be a sublattice of $\ZZ^d$ generated by $d$ independent elements $n^{(i)}, i\in[1,d]$, so that $\ZZ^d/\Lambda$ is a finite group.  One has the following exact sequences of discrete groups and their duals in the sense of homomorphisms into $\CC^*$.
\begin{equation*}
  \renewcommand{\arraystretch}{1.1}
\left.
\begin{array}{ccccccc}
  \ZZ^d &\cong& \Lambda &\hookrightarrow& \ZZ^d &\twoheadrightarrow& \ZZ^d/\Lambda \\
   (\CC_w^*)^d &\cong& (\CC_z^*)^d/G &\twoheadleftarrow& (\CC_z^*)^d &\hookleftarrow& G
\end{array}
\right.
\end{equation*}
The bottom groups are (isomorphic to) the duals of the upper groups.  An element of $(\CC_z^*)^d$ as the dual of $\ZZ^d$ is the weight vector of the character, or the Floquet multiplier vector.  $G$~is the kernel of the restriction of a character of $\ZZ^d$ to~$\Lambda$ (it is isomorphic to the finite group $\Lambda^*/\ZZ^d$).  The injection $\ZZ^d \hookrightarrow \ZZ^d$ onto $\Lambda$ in the upper row is given by
\begin{equation}
  e^{(i)} \mapsto n^{(i)},
\end{equation}
in which $e^{(i)}$ is the $i$-th elementary vector in $\ZZ^d$.  The dual map $(\CC_w^*)^d \twoheadleftarrow (\CC_z^*)^d$ in the second row is given by the monomial map~$M$
%
%
%
\begin{equation}
  Mz \;:=\; (z^{n^{(1)}},\dots,z^{n^{(d)}}) \;=\; (w_1,\dots,w_d) \;=\; w \;\;\mapsbackto\;\; z \;=\; (z_1,\dots,z_d).
\end{equation}

\subsubsection{Relating the Fermi varieties} 
Let $\EE_z$ denote the eigenspace of the (original) $\ZZ^d$ action for weight $z\in(\CC_z^*)^d$, and let $\tilde\EE_w$ denote the eigenspace of the $\ZZ^d\cong\Lambda$ action for weight $w\in(\CC_w^*)^d$.  The space $\tilde\EE_w$ is invariant under the original $\ZZ^d$ action and is therefore decomposed as
\begin{equation}
  \tilde\EE_w \;=\; \bigoplus\limits_{\zeta\in G} \EE_{\zeta z},
\end{equation}
in which $z$ is such that $Mz=w$; note that $\{ \zeta z : \zeta\in G\}=M^{-1}(w)$.  Since $A$ commutes with~$\ZZ^d$,
we obtain the spectral union
\begin{equation}\label{spectraldecomposition}
  \sigma(A|_{\tilde\EE_w}) \;=\; \bigcup\limits_{\zeta\in G} \sigma(A|_{\EE_{\zeta z}}).
\end{equation}

Let $D(z,\lambda)$ be the dispersion function for $A$ as a periodic operator with respect to the original $\ZZ^d$ action; and let $\tilde D(w,\lambda)$ be the dispersion function for $A$ as a periodic operator with respect to the $\ZZ^d\cong\Lambda$ action.  The associated Fermi varieties are
\begin{equation}
  \Phi_\lambda = \{ z\in(\CC_z^*)^d : D(z,\lambda)=0 \},
  \quad
  \tilde \Phi_\lambda = \{ w\in(\CC_w^*)^d : \tilde D(w,\lambda)=0 \}.  
\end{equation}
The coordinate rings are
\begin{equation}
  R_\lambda = \frac{\CC[z_1,\dots,z_d]}{\langle D(z,\lambda) \rangle},
  \qquad
  \tilde R_\lambda = \frac{\CC[w_1,\dots,w_d]}{\langle \tilde D(w,\lambda) \rangle}.
\end{equation}
We now prove that $M$ maps $\Phi_\lambda$ to $\tilde \Phi_\lambda$ surjectively.  That the image of $\Phi_\lambda$ under $M$ lies in $\tilde\Phi_\lambda$ follows from
\begin{equation*}
  D(z,\lambda) = 0 \implies
  \lambda\in\sigma(A|_{\EE_z}) \implies
  \lambda\in\sigma(A|_{\tilde\EE_{Mz}}) \implies
  \tilde D(Mz,\lambda)=0.
\end{equation*}
To show that $M$ maps $\Phi_\lambda$ onto $\tilde\Phi_\lambda$, let $w\in(\CC^*)^d$ be such that $\tilde D(w,\lambda)=0$, so that $\lambda\in\sigma(A|_{\tilde\EE_w})$.  Because of (\ref{spectraldecomposition}) there exists $z:Mz=w$ and $\lambda\in\EE_z$, and thus $D(z,\lambda)=0$.
We conclude that
\begin{equation}
  \big\{ z : \tilde D(Mz,\lambda)=0 \big\} \;=\; \big\{ z : \prod_{\zeta\in G} D(\zeta z,\lambda)=0 \big\}.
\end{equation}
The Nullstellensatz tells us that these two Laurent polynomials in $z$ have the same irreducible factors up to multiplicity.
In fact, the multiplicities are the same.
This may be shown using more subtle and delicate arguments and definitions than we chose to present. 

We have the following surjection of varieties and injection of coordinate rings:
\begin{equation*}
  \renewcommand{\arraystretch}{1.1}
\left.
\begin{array}{ccccc}
  M: & \Phi_\lambda & \twoheadrightarrow & \tilde\Phi_\lambda & \\
      & R_\lambda & \hookleftarrow & \tilde R_\lambda & :M^*
\end{array}
\right.
\end{equation*}
Since $M^*$ is injective, $\tilde R_\lambda$ is an integral domain whenever $R_\lambda$ is; thus $\tilde\Phi_\lambda$ is irreducible whenever $\Phi_\lambda$~is.  Let us now show why $\tilde\Phi_\lambda$ is reducible whenever $\Phi_\lambda$~is.  Let $D_1(z,\lambda),\dots,D_r(z,\lambda)$ be the distinct irreducible factors of $D(z,\lambda)$.  Then for some positive powers~$p_i$,
\begin{equation}
  \prod_{\zeta\in G} D(\zeta z,\lambda)
  \;=\; \prod_{\zeta\in G} D_1^{p_1}(\zeta z,\lambda) \cdots \prod_{\zeta\in G} D_r^{p_r}(\zeta z,\lambda).
\end{equation}
Thus, $\tilde D(Mz,\lambda)$ has each of $\prod_{\zeta\in G}D_i(\zeta z,\lambda)$ as a factor, but since the latter is invariant under the $G$ action $z\to\zeta z$, there exists a Laurent polynomial $\tilde D_i(w,\lambda)$ such that $\prod_{\zeta\in G}D_i(\zeta z,\lambda)=\tilde D_i(Mz,\lambda)$.  The remaining factor of $\tilde D(Mz,\lambda)$ is also invariant under $G$, and we obtain a factorization of~$\tilde D(w,\lambda)$, which makes $\tilde R_\lambda$ reducible.

\section{Nondegeneracy of band edges and beyond}
A property of the spectrum revealed by the Bloch variety is its behavior near edges of spectral bands.
The  spectral edges conjecture about this behavior  leads to a deeper study of Bloch varieties.

\subsection{Spectral edges nondegeneracy conjecture}\label{S:SENDC}

The real Bloch variety $\RR\BB_{\!A}\subset\TT^d\times\RR$ of a periodic graph
operator $A$ encodes the relation between the characters $z\in\TT^d$ of the $\ZZ^d$-action and the spectrum $\lambda\in\RR$
of the operator ({\it cf.} \S\ref{sec:BlochFermi}).
Its projection to $\RR$ is the spectrum $\sigma(A)$ of $A$ which consists of intervals (spectral bands) each of which is an
image of one of the $|W|$ branches (band functions) of $\RR\BB_{\!A}$ over $\TT^d$.
Endpoints (spectral edges) of spectral bands are the images of extrema of the corresponding band function of $\RR\BB_{\!A}$.
Many important notions in physics, including effective mass in solid state physics, the Liouville property, Green's
function asymptotics, Anderson localization, and homogenization, require 
that those extrema are nondegenerate in that the Hessian matrix has full rank.
Kuchment noted that this nondegeneracy is often assumed without proof 
and posed the
{\itshape spectral edges conjecture}~\cite[Conj.\ 5.25]{Kuchment2016}.
This posits that for generic parameters (potential and edge labels), each extreme value is attained by a single band function and 
the extrema are all isolated and nondegenerate.
It is stated and discussed when $d=2$ for continuous operators in~\cite{ParnovskiShterenber2017a}.

While posed for all periodic operators (discrete and {\itshape continuous}), it remains largely open, even for discrete
operators.
In this setting, ``generic'' means avoiding an algebraic subset of the parameters.
This conjecture does not always hold.
We have already seen an instance of this in Figure~\ref{F:AA_and_AB}.
The simplest nontrivial counterexample is due to Filonov and Kachkovskiy in~\cite[\S7]{FiK}.
We give a slightly more involved example.
Figure~\ref{fig:FiK} shows a $\ZZ^2$-periodic graph with five (orbits of) edges, a labeling, and a Bloch variety when the
\begin{figure}[htb]
  \centering
   \begin{picture}(122,92)
     \put(0,0){\includegraphics{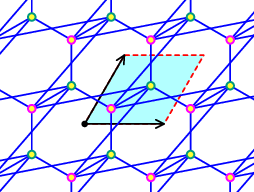}}
   \end{picture}
   \qquad
   \begin{picture}(153,108)(-23,-9)
     \put(-22,-15){\includegraphics{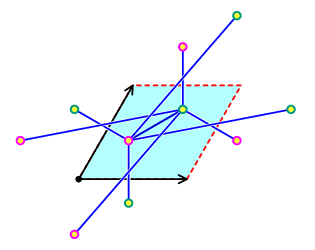}}
     
    \thicklines 
      \put(66.7,35.7){{\color{white}\vector(-1,1){10}}} \put(67.3,36.3){{\color{white}\vector(-1,1){10}}}
      \put(67.9,35.1){{\color{white}\vector(-1,1){11}}}

     \thinlines
     \put(32,32){\small$u$}   \put(100,18){\small$u\dotplus x$}  \put(-22.5,29){\small$u\dotminus x$}
                              \put( 33,81){\small$u\dotplus y$}  \put( -9,-9){\small$u\dotminus y$}
     \put(105,25){\vector(-1,1){10}}     \put(50,83){\vector(1,0){12}}
     \put(69,54){\small$v$}   \put( -9,51.5){\small$v\dotminus x$}  \put(110.5, 57){\small$v\dotplus x$}
                              \put( 30,-2){\small$v\dotminus y$}  \put(95, 92){\small$v\dotplus y$}
     \put(59,22){\small$x$}
     \put(32,60){\small$y$}


     \put(67.5,31){\small{\color{blue}$a$}}   \put(67.6,35.4){\vector(-1,1){10}}
     \put(42,13){\small{\color{blue}$c$}}   \put(59,72 ){\small{\color{blue}$c$}}
     \put(26,14){\small{\color{blue}$e$}}   \put(78,76  ){\small{\color{blue}$e$}}
     \put(21,50){\small{\color{blue}$b$}}   \put(83,32 ){\small{\color{blue}$b$}}
     \put(14,34){\small{\color{blue}$d$}}   \put(93,50.5){\small{\color{blue}$d$}}
   \end{picture}
   \qquad
   \includegraphics[height=100pt]{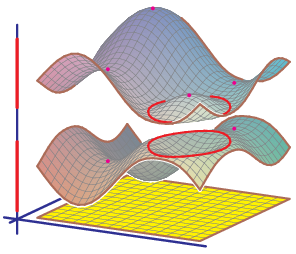}
   \caption{Filonov-Kachkovskiy Bloch variety with a curve of critical points.}
  \label{fig:FiK}
\end{figure}
potential is not constant, but $b=d$ and $c=e$.
This has eight isolated critical points and two curves of critical points (the corresponding Fermi variety is
defined by $a+b(x+x^{-1}) + c(y+y^{-1})$), each lying over the  edges of the spectral gap.
In particular, for pure Schr\"odinger operators ($a=\dotsb=e=1$) the spectral edges nondegeneracy conjecture fails.
When $b\neq d$ or $c\neq e$, all critical points are isolated.

About the same time as~\cite{FiK}, Parnovski and Shterenberg found higher-dimensional examples~\cite[Rem.\ 4.1]{ParnovskiShterenber2017a},
as did Filonov and Kachkovskiy~\cite{FilonovKachkovski2024a}.
Flat bands (e.g.\ the Lieb lattice~\eqref{Eq:Lieb}) give other examples with non-isolated critical points.
These examples all involve algebraic subsets of the full set of parameters.
For general values of the parameters, all critical points are isolated and nondegenerate.
A more subtle counterexample is the middle Bloch variety of~\eqref{Eq:BlochVars}---over each edge of the gap there are two
nondegenerate critical points.
(A version of the spectral edges conjecture posits that each extrema is attained at a unique point on the Bloch variety.)
This occurs whenever the parameters $a,b,c$ (see~\eqref{Eq:hexagonal_lattice}) are the sides of a triangle~\cite{BlochDiscriminants}.

For a positive result, Liu~\cite{Liu22} proved that the extrema are isolated for the Schr\"{o}dinger operator acting on the
square lattice.

\subsection{Nondegeneracy of critical points}
The Spectral Edges Conjecture is discussed in~\cite{DoKuchmentSottile2019a}, which introduces a
computational method to study it.
Extrema of bands of the real Bloch variety are some (but not all) of the critical points of the coordinate function $\lambda$ on
the complex Bloch variety $\BB_A$.
A strengthening of the spectral edges conjecture is the {\itshape nondegeneracy conjecture}---that every critical point
of $\lambda$ on $\BB_A$ is nondegenerate.
We have the following dichotomy.

\begin{theorem}[{\protect\cite[Thm.~12]{DoKuchmentSottile2019a}}]
  For a given graph $\Gamma$, there is a dense open subset $U$ in the space of parameters for $\Gamma$ with the
  property that either all operators on $\Gamma$ with parameters in $U$ have all critical points nondegenerate,
  or every such operator has degenerate critical points.
\end{theorem}

This is because there is a system of polynomial equations~\eqref{Eq:CPE} involving the parameters that define the critical points
(the critical point equations) and a further equation (determinant of the Hessian) that defines the degenerate critical points.
Consequently, the set $V$ of of parameters whose corresponding Bloch variety has a degenerate critical point
is constructible.
If $V$ is not dense, then $U$ is an open subset of its complement and operators in $U$ have all critical points nondegenerate.
If $V$ is dense, then  $U$ is an open subset of $V$ and generic operators have degenerate critical points.

A recent preprint~\cite{FS25} shows that all critical points are nondegenerate and occur only at points of symmetry
$z^2=1$ in the following cases.
(1) The graph parameters are generic in the above  and dispersion polynomial $D(z,\lambda)$ sufficiently sparse in that
as a polynomial in $z$, the only monomials are $z_i$ and $z_i^{-1}$.
(2) The edge parameters are fixed but the potential is generic and $\Gamma$ is minimally connected, in a sense explained
in~\cite{FS25}.

\subsection{Critical point degree}

The algebraic nature of the set of critical points and standard facts in algebraic geometry imply that given a graph
$\Gamma$, there is a number $N(\Gamma)$ and a open dense subset $U$ of the parameters for $\Gamma$ with the following properties:
Every operator on $\Gamma$ has at most $N(\Gamma)$ isolate critical points,
and if the operator has parameters from $U$, then it has has exactly $N(\Gamma)$ critical points.
This number $N(\Gamma)$ is called the {\itshape critical point degree} of~$\Gamma$~\cite{FRS}.

As explained in~\cite[\S5]{DoKuchmentSottile2019a}, if there exists one operator on $\Gamma$ with $N(\Gamma)$ nondegenerate
critical points, then the nondegeneracy conjecture holds for $\Gamma$.
Consequently, if $N(\Gamma)$ is known, then a single computation can  establish the nondegeneracy conjecture for $\Gamma$.
This is used to prove~\cite[Thm.~16]{DoKuchmentSottile2019a} and to establish the nondegeneracy conjecture for over $2^{19}$ graphs
in~\cite{FaustSottile2023a}.

The paper~\cite{FaustSottile2023a} studies the critical points of a periodic operator, giving bounds and conditions which imply the
bounds are met.
A standard formulation of critical points of implicit functions from algebraic optimization gives the following system of equations for the critical points of $\lambda$ on the Bloch variety:
 \begin{equation}\label{Eq:CPE}
  D(z,\lambda) \;=\;
  z_1 \frac{\partial D}{\partial z_1}(z,\lambda)
  \;=\; \dotsb \;=\;
  z_d \frac{\partial D}{\partial z_d}(z,\lambda) \;=\; 0\,.
 \end{equation}
 (As $z\in(\CC^\times)^d$, multiplying by $z_i$ does not add solutions.)
 A monomial $z^n\lambda^j$ is an eigenvector for the Euler operator $z_i\frac{\partial }{\partial z_i}$ with eigenvalue $n_i$.
 Thus, the exponents $(n,j)$ of monomials $z^n\lambda^j$ that occur in each of the
 {\itshape Critical Point Equations}~\eqref{Eq:CPE} form a subset of those 
 that occur in $D(z,\lambda)$, called its {\itshape support}.

\subsection{Toric compactification}
Let $\calN(A)\subset\RR^{d+1}$ be the convex hull of the support of the dispersion function $D(z,\lambda)$ for the operator $A$.
This is the Newton polytope of $D(z,\lambda)$.
Figure~\ref{Fig:polytopes} shows Newton polytopes of general operators on the hexagonal and Lieb lattices, the graph of Figure~\ref{fig:FiK},
and another graph with $|W|=3$.
\begin{figure}[htb]
 \centering
   \includegraphics[height=80pt]{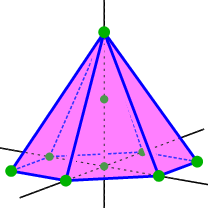}\quad 
   \includegraphics[height=110pt]{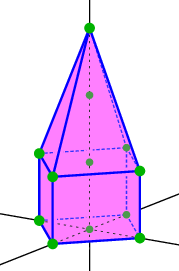}\quad 
   \includegraphics[height=80pt]{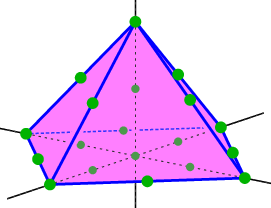}\quad 
   \includegraphics[height=110pt]{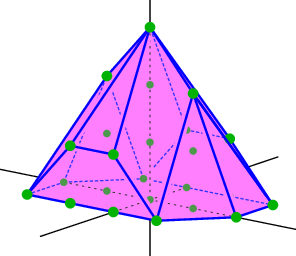}
   \caption{Four Newton polytopes}
   \label{Fig:polytopes}
 \end{figure}
A classical result of Kushnirenko~\cite{Kushnirenko} gives the following bound for the number of critical points.

\begin{theorem}[{\protect\cite[Cor.~2.5]{FaustSottile2023a}}]\label{Th:FS_A}
  The number of isolated critical points of the function $\lambda$ on the Bloch variety $\BB_A$ is at most
  $(d+1)!\mathrm{vol}(\calN(A))$.
\end{theorem}

An analysis of the proof (or of~\cite[Thm.~B]{Bernstein}) gives conditions for this bound to be sharp.
Given a face $F$ of the Newton polytope $\calN(A)$, the (sum of) the terms of $D(z,\lambda)$ whose monomials lie on $F$ is the
{\itshape facial form} $D_F(z,\lambda)$---this is a strict generalization of the notion of terms of highest degree.
Similarly, the {\itshape facial subsystem} $\mbox{(CPE)}_F$ of the Critical Point Equations~\eqref{Eq:CPE} is the
collection of facial forms of the equations in~\eqref{Eq:CPE}.
The following is a consequence of \cite[Cor.~3.5]{FaustSottile2023a}.

\begin{theorem}
  \label{Th:NP}
  The inequality in  Theorem~\ref{Th:FS_A} is an equality if and only if for every face $F$ of $\calN(A)$ that is not
  its base, the facial subsystem  $\mbox{(CPE)}_F$ has no solutions.
\end{theorem}

Solutions to a facial subsystem $\mbox{(CPE)}_F$ are {\itshape asymptotic critical points}.
These are studied in~\cite{FRS, FaustSottile2023a}.
If a face $F$ is vertical (as in the second and fourth polytopes in Figure~\ref{Fig:polytopes}, then $\mbox{(CPE)}_F$ has solutions.
For a face $F$ that is not vertical, there are asymptotic critical points if and only if the variety
defined by $D_F(z,\lambda)$ is singular.
In~\cite{FRS} such asymptotic critical points for general operators on a graph $\Gamma$ are shown to arise from structural
properties of $\Gamma$.

\subsection{More on toric compactification}

The introduction and use of facial forms and subsystems in~\cite{FRS, FaustSottile2023a}, as well as the use of the component of
$D(z,\lambda)$ of lowest degree in~\cite{FillmanLiuMatos2022}, are all algebraic manifestations of a natural toric compactification
of $\BB_A\subset(\CC^\times)^d\times\CC$.
This notion was introduced by Gieseker, Kn\"orrer, and Trubowitz~\cite{GiesekerKnorrerTrubowitz1993} (see also~\cite{Peters1990a})
in their study of the square lattice $\ZZ^2$ under the free action of $a\ZZ\oplus b\ZZ$, and extended by
B\"attig~\cite{Battig1988,Battig1992}.
We sketch a modern view developed in~\cite{FaustLopezShipmanSottile}.

Given a polytope $\calN\subset\RR^r$ with vertices in $\ZZ^r$, there is a projective toric variety~\cite{CLS} $X_\calN$ lying in
a projective space $\PP(\calN)$ whose coordinates correspond to the integer points in $\calA\vcentcolon=\calN\cap\ZZ^r$.
(The space $\PP(\calN)$ is the quotient of $\CC^\calA\smallsetminus\{0\}$  by scalars.
It is a compact complex manifold.)
There is a map $\varphi\colon(\CC^\times)^r\to\PP(\calN)$ that sends a point $x\in(\CC^\times)^r$ to the vector
$(x^a\mid a\in\calA)$ of monomials, and $X_\calN$ is the closure of its image.
This has the important geometric consequence (which is the idea behind Kushnirenko's Theorem) that under $\varphi$, linear functions
on $\PP(\calN)$ correspond to polynomials on $(\CC^\times)^r$ with support in $\calN$, and the degree of $X_\calN$ is
$r!\vol(\calN)$.
Lastly, the difference $X_\calN\smallsetminus\varphi((\CC^\times)^r)$ is a union of toric subvarieties $X_F$,
one for each face $F$ of~$\calN$.

In our context, let $\overline{\BB_A}$ be the closure of $\varphi(\BB_A)$ in $X_{\calN(A)}$.
Then for each face $F$ of $\calN(A)$,  the facial form $D_F(z,\lambda)$ defines $\overline{\BB_A}\cap X_F$,
which we regard as the asymptotic Bloch variety along $X_F$.
Similarly, asymptotic critical points are those that lie along some $X_F$, for $F$ not the base of $\calN(A)$
(the base corresponds to $\lambda=0$).

A first step to generalize~\cite{Battig1988,Battig1992,GiesekerKnorrerTrubowitz1993} is to extend the vector bundle $\scrE$
to $X_{\calN(A)}$---it remains a trivial bundle.
The endomorphism $\hat{A}(z)-\lambda I_W$ similarly extends to the bundle $\scrE$ over $X_{\calN(A)}$ with kernel the Floquet sheaf
$\scrF$ on $X_{\calN(A)}$ whose support is the compactified Bloch variety  $\overline{\BB_A}$.
This may always be done.

In his work on the square lattice, B\"attig~\cite{Battig1988,Battig1992} does more.
To every face $F$ of $\calN(A)$ he associates a periodic, labeled, directed graph $\Gamma_F$  whose operator $A_F$ has Bloch variety equal to $\overline{\BB_A}\cap X_F$.
The paper~\cite{FaustLopezShipmanSottile} investigates to what extent this may be done for a general graph $\Gamma$.  When $\calN(A)$ is the pyramid $|W|\cdot P$, where $P$ is the convex hull of all entries in the matrix  $\hat{A}(z)-\lambda I_W$, then the full package of B\"attig---asymptotic spectral problems for each face of $\calN(A)$---holds.  Only the third Newton polytope in Figure~\ref{Fig:polytopes} has this form.

A compactification that does not quite fit this mode was given for a family of periodic operators
in~\cite{BattigKnorrerTrubowitz1991}.
It is an open problem to understand this in greater generality.

\vspace{3ex}

\noindent
{\bfseries Acknowledgment.} This material is based upon work supported by the National Science Foundation under Grant
No. DMS-2206037 (SPS), and Grant No.\ DMS-2201005 (FS).


\bibliography{AAPGO}

\end{document}